\documentclass[12pt]{amsart}

\usepackage{tgpagella}
\usepackage{euler}
\usepackage[T1]{fontenc}
\usepackage{amsmath, amssymb}
\usepackage[hidelinks]{hyperref}
\usepackage[english]{babel}
\usepackage{mathrsfs}
\usepackage{eucal}
\usepackage[all]{xy}
\usepackage{tikz}

\newtheorem{thm}{Theorem}[subsection]
\newtheorem*{thm*}{Theorem}
\newtheorem{lem}[thm]{Lemma}
\newtheorem{fact}[thm]{Fact}

\newtheorem{prop}[thm]{Proposition}
\newtheorem*{prop*}{Proposition}
\newtheorem{conj}[thm]{Conjecture}
\newtheorem{cor}[thm]{Corollary}
\newtheorem*{cor*}{Corollary}

\newtheorem*{JT}{Jung's theorem}
\theoremstyle{definition}
\newtheorem{defn}[thm]{Definition}
\newtheorem*{defn*}{Definition}

\newtheorem{remark}[thm]{Remark}

\newtheorem{question}[thm]{Question}
\newtheorem{example}[thm]{Example}
\newtheorem{examples}[thm]{Examples}

\newtheorem*{question*}{Question}
\newtheorem*{Pquestion*}{Popa's question}
\newtheorem{conv}[thm]{Convention}
\newtheorem*{conv*}{Convention}

\def\bb{\mathbb}

\def\bb{\mathbb}

\def\u{\mathsf 1}

\newcommand{\cstar}{$\mathrm{C}^*$}

\def \tp{\operatorname{tp}}

\makeatletter

\def\dotminussym#1#2{%
  \setbox0=\hbox{$\m@th#1-$}%
  \kern.5\wd0%
  \hbox to 0pt{\hss\hbox{$\m@th#1-$}\hss}%
  \raise.6\ht0\hbox to 0pt{\hss$\m@th#1.$\hss}%
  \kern.5\wd0}
\newcommand{\dotminus}{\mathbin{\mathpalette\dotminussym{}}}

\DeclareMathOperator{\tr}{tr}

\def \R{\mathcal R}
\def \u{\mathcal U}

\newcommand{\mc}{\mathcal}
\newcommand{\mb}{\mathbb}
\newcommand{\HOM}{\mb{H}\text{om}}
\newcommand{\tql}{\textquotedblleft}
\newcommand{\tqr}{\textquotedblright}
\newcommand{\ds}{\displaystyle}

\def\McDuff{(0,0) ellipse (9.3 and 4.1)}
\def\PGamma{ (0,0) ellipse (10 and 5)}
\def\EER{ (1.35,0) ellipse (1.5 and 1.1)}
\def\SuperMcDuff{ (3.7,0) ellipse (4.3 and 2.7)}
\def\StronglyMcDuff{ (5.95,0) ellipse (1.6 and 1.1)}
\def\AC{ (-.4,0) ellipse (4.4 and 3.0)}
\def\IAC{ (-.4,0) ellipse (3.7 and 2.3)}
\def\WeaklyMcDuff{ (-5.5,0) ellipse (3.0 and 1.4)}
\def\EC {(-.9,0) ellipse (2 and 1.1)}

\makeatletter
\def\l@subsection{\@tocline{2}{0pt}{2.5pc}{5pc}{}}
\def\l@subsubsection{\@tocline{2}{0pt}{5pc}{7.5pc}{}}
\makeatother

\begin{document}


\title[Factorial relative commutants and the generalized Jung property]{Factorial relative commutants and the generalized Jung property for II$_1$ factors}
\author{Scott Atkinson, Isaac Goldbring and Srivatsav Kunnawalkam Elayavalli}
\thanks{I. Goldbring was partially supported by NSF CAREER grant DMS-1349399.}

\address{Department of Mathematics\\University of California, Riverside, Skye Hall, 900 University Ave., Riverside, CA 92521 }
\email{scott.atkinson@ucr.edu}
\urladdr{https://sites.google.com/view/scottatkinson}

\address{Department of Mathematics\\University of California, Irvine, 340 Rowland Hall (Bldg.\# 400),
Irvine, CA 92697-3875}
\email{isaac@math.uci.edu}
\urladdr{http://www.math.uci.edu/~isaac}

\address{Department of Mathematics\\Vanderbilt University, 1326 Stevenson Center, Station B 407807, Nashville, TN 37240}
\email{srivatsav.kunnawalkam.elayavalli@vanderbilt.edu}
\urladdr{https://sites.google.com/view/srivatsavke}
\begin{abstract}




The findings reported in this paper aim to garner the interest of both model theorists and operator algebraists alike. Using a novel blend of model theoretic and operator algebraic methods, we show that the family of II$_1$ factors elementarily equivalent to the hyperfinite II$_1$ factor $\R$ all admit embeddings into $\R^\u$ with factorial relative commutant. This answers a long standing question of Popa for an uncountable family of II$_1$ factors. We introduce the notion of a generalized Jung factor: a II$_1$ factor $M$ for which any two embeddings of $M$ into its ultrapower $M^\u$ are equivalent by an automorphism of $M^\u$. As an application of the result above, we show that $\R$ is the unique $\R^\u$-embeddable generalized Jung factor. Using the concept of building von Neumann algebras by games and the recent refutation of the Connes embedding problem, we also show that there exists a generalized Jung factor which does not embed into $\R^\u$. Moreover, we find that there are uncountably many non $\R^\u$-embeddable generalized Jung type II$_1$ von Neumann algebras. We study the space of embeddings modulo automorphic equivalence of a II$_1$ factor $N$ into an ultrapower II$_1$ factor $M^\u$ and equip it with a natural topometric structure, yielding cardinality results for this space in certain cases. These investigations are naturally connected to the super McDuff property for II$_1$ factors: the property that the central sequence algebra is a II$_1$ factor. We provide new examples, classification results, and assemble the present landscape of such factors. Finally, we prove a transfer theorem for inducing factorial commutants on embeddings with several applications. 

\end{abstract}

\maketitle

\tableofcontents
\pagebreak
\section*{Introduction}





A fundamental philosophy in mathematics is the idea that one can deduce structural characteristics of a given object by embedding it into a richer space with tractable structure.  
In the present article, we work in the context of embeddings of II$_1$ factor von Neumann algebras into ultrapowers of II$_1$ factors.  In particular, given a II$_1$ factor $N$, it is of significant interest to extract structural properties of $N$ by examining how $N$ embeds into $\R^\u{\,}$ (an ultrapower of the separably acting II$_1$ factor $\R$) and how $N$ embeds into its own ultrapower $N^\u$. (See \S \ref{vnaprelim} for the relevant definitions.)
We will say that a II$_1$ factor is \textbf{embeddable} if it embeds into $\R^\u$. With the recent negative resolution of the Connes Embedding Problem\footnote{This asked if every separable II$_1$ factor can be embedded into $\R^\u$.} in \cite{connessol}, embeddability is a nontrivial assumption.

A good starting point for our context is the following standard fact: any two embeddings of $\R$ into $\R^{\u}$ are unitarily equivalent.\footnote{This fact is well-known in the II$_1$ factor community; for a reference, see Proposition 1.7 of \cite{ultraprodembed} and Theorem 3.1 of \cite{autoultra}.} In \cite{jung} Jung established the striking result that the converse of the previous statement holds:
\begin{JT}[\cite{jung}]\label{jung's theorem}
Any two embeddings of an embeddable II$_1$ factor $N$ into $\R^{\u}$ are unitarily equivalent if and only if $N\cong \R$. 
\end{JT}
This, combined with the seminal result of Connes in \cite{connes}, tells us that the structural property of \textbf{amenability} is precisely captured by the space of embeddings of $N$ into $\R^{\u}$ modulo unitary equivalence.


Naturally, there are many generalizations and variants of Jung's theorem in the literature. The following is an immediate corollary from \cite{ultraprodembed}: If $N$ is an embeddable II$_1$ factor such that any two embeddings of $N$ into $N^\u$ are unitarily equivalent, then $N\cong \R$. While it does follow from the general ultraproduct codomain result from \cite{ultraprodembed} (namely Corollary 3.8), this special case does not require the same technical machinery. In fact, the proof of this special case is much simpler than the proof of Jung's theorem itself, the salient point being the availability of the diagonal embedding of $N$ into $N^\u$. 

To make the connection between the results above and the main results of this paper, we make the following definitions.

\begin{enumerate}
    \item For II$_1$ factors $M$ and $N$, we say $(N,M)$ is a \textbf{Jung pair} if $N$ embeds into $M^\u$ and any two embeddings of $N$ into $M^u$ are unitarily equivalent. We say $M$ has the \textbf{Jung property} if $(M,M)$ is a Jung pair.

    \item For II$_1$ factors $M$ and $N$, we say $(N,M)$ is a \textbf{generalized Jung pair} if $N$ embeds into $M^\u$ and any two embeddings of $N$ into $M^u$ are automorphically equivalent. We say $M$ has the \textbf{generalized Jung property} if $(M,M)$ is a generalized Jung pair.
\end{enumerate}

Thus Jung's theorem states that $(N,\R)$ is a Jung pair if and only if $N\cong \R$, and the variation mentioned above says that an embeddable II$_1$ factor $N$ has the Jung property if and only if $N \cong \R$.  One of the main results of the present article is a strong improvement of the latter characterization as follows:

\begin{thm*}\label{usmain}
 If $N$ is an embeddable II$_1$ factor, then $N$ has the generalized Jung property if and only $N\cong \R$.
 \end{thm*}

While the characterization of embeddable Jung factors is a result on embeddings modulo \emph{inner} automorphisms of the ultrapower codomain, the above theorem addresses equivalence of embeddings modulo \emph{all} automorphisms.\footnote{At this moment, we should also mention the group theoretic analog of these considerations. In \cite{elekszabo}, Elek and Szabo  proved a Jung-type theorem for sofic groups. P$\breve{\text{a}}$unescu also asks in \cite{paunescuauto} about the case when one considers arbitrary automorphisms of the ultrapower.
Also, it should be noted that in \cite{paunescu1} and \cite{paunescu2} P$\breve{\text{a}}$unescu developed the theory of the convex structure of sofic approximations in the spirit of Brown in \cite{topdyn} as described below.} We go further and show that the space of embeddings modulo automorphic equivalence has a natural topometric structure.  As a consequence, we obtain the result that there are uncountably many automorphic equivalence classes of embeddings for certain families of non-amenable II$_1$ factors--see \S \ref{HOM}.

The second main result addresses the non-embeddable case:

\begin{thm*}
There is a non-embeddable II$_1$ factor with the generalized Jung property.
\end{thm*}

We conjecture that there are continuum many non-isomorphic separable II$_1$ factors with the generalized Jung property and we present some mild evidence to support this.  In particular, we show that there are continuum many non-isomorphic separable type II$_1$ von Neumann algebras with the generalized Jung property--see \S\S \ref{generalcase}.

At this point, we need to bring some set theory into the picture. Farah showed in \cite[Corollary 16.7.2]{farah} that if one assumes the Continuum Hypothesis, then every ultrapower II$_1$ factor $N^\u$ has an automorphism that does not lift to an automorphism on $\ell^\infty(N)$.\footnote{The reference given discusses the case of \cstar-algebras, but the case of tracial von Neumann algebras is identical.}  Consequently, in the presence of the Continuum Hypothesis the equivalence relation of automorphic equivalence for embeddings into an ultrapower is indeed coarser than that of unitary equivalence.  Thus, we adhere to the following convention:

\begin{conv*}
Throughout this paper, we assume the Continuum Hypothesis (CH) holds.\footnote{The standing CH assumption will also allow us to explain some model-theoretic notions in a language that should appeal more to operator algebraists.} 
\end{conv*}

The convex structure $\HOM(N,\R^\u)$ introduced by Brown in \cite{topdyn} (and later generalized to $\HOM(N,M^\mc{U})$ in \cite{brocap} by Brown and Capraro and in \cite{saa} by the first-named author) plays a significant role in the development of the main result.  In particular, we make use of the following undersung\footnote{Brown's wording: \tql Though we won't need it, here's a cute consequence.\tqr} characterization of $\R$ from \cite{topdyn}:  a separable embeddable II$_1$ factor is hyperfinite if and only if every embedding of it into $\R^\mc{U}$ has factorial relative commutant. The reader familiar with these convex spaces is aware of their connection with a long-standing open problem due to Popa:

\begin{Pquestion*}
Does every separable embeddable II$_1$ factor admit an embedding into $\R^\u$ with a factorial relative commutant?
\end{Pquestion*}


In the proof of the main results of this article, we make noteworthy progress on Popa's question.
At the time of the writing of this paper, only a handful of examples of II$_1$ factors are known to satisfy the conclusion of Popa's question, e.g., $\R$ and $L(SL_n(\bb{Z}))$ for $n \geq 3$, odd. We provide continuum many pairwise non-isomorphic II$_1$ factors which satisfy the conclusion of Popa's question:

\begin{thm*}
If $M$ is a II$_1$ factor \textbf{elementarily equivalent} to $\R$, then every \textbf{elementary} embedding $M\hookrightarrow \R^\u$ has factorial commutant.\footnote{The notions appearing in bold are model-theoretic terms that will be defined in \S\ref{mtprelim}.}
\end{thm*}

This consequence also sheds light on so-called \tql super McDuff\tqr \, II$_1$ factors. Recall that a II$_1$ factor is said to be McDuff if the relative commutant  $M'\cap M^\u$ is non abelian. If the commutant $M'\cap M^\u$ is moreover a II$_1$ \emph{factor}, we say that $M$ is \textbf{super McDuff}.  (This notion was first considered by Dixmier and Lance in \cite{DL} but not given a name until the article \cite{BCIexplained} by the second-named author and Hart).  Dixmier and Lance proved that $\R$ is super McDuff.  Before the writing of this paper, there were only a few more known examples of super McDuff factors.  The above theorem yields continuum many pairwise non-isomorphic separable super McDuff factors:

\begin{cor*}
Any separable II$_1$ factor elementarily equivalent to $\R$ is super McDuff.
\end{cor*}

We close the paper with some results on how the property of having a factorial relative commutant can be transferred from one embedding to another via composition. These results are motivated by the general interest of Popa's question and the direct connection this property has with the notion of generalized Jung pairs.  In particular, we obtain the following significant upgrade of Brown's characterization of $\R$:

\begin{thm*}
Let $N$ be a separable embeddable II$_1$ factor and let $M$ be a II$_1$ factor that is either McDuff or embeddable.  If every embedding of $N$ into $M^\mc{U}$ has factorial relative commutant, then $N \cong \R$.
\end{thm*}

The paper is organized as follows. In \S\ref{vnaprelim} and \S\ref{mtprelim}, we provide (respectively) the relevant operator algebraic and model theoretic preliminaries. \S\ref{jungprop} discusses some observations and obtains the main results on the generalized Jung property (Theorems \ref{nonembedGJ} and \ref{yay}). We introduce the topometric structure on the space of automorphic equivalence classes of embeddings in \S\ref{HOM}, obtaining partial cardinality results. In \S\ref{genpairs}, we include a brief discussion on generalized Jung pairs of II$_1$ factors, presenting some open questions. In \S\ref{relres}, we collect past results from across the literature together with new results on super McDuff factors.  \S \ref{Tran} presents the transfer theorem for factorial relative commutants and exhibits numerous applications.

\subsection*{Acknowledgments} This work was initiated when the three authors met at the Banff International Research Station for the workshop \tql Classification Problems in von Neumann algebras\tqr\ during October 2019. We thank BIRS and the organizers for hosting a wonderful conference.  We would also like to thank  Bradd Hart and David Sherman for helpful conversations about this project and Adrian Ioana for allowing us to include his proof of Proposition \ref{adrian}. 


\section{Tracial von Neumann algebras and their ultraproducts}\label{vnaprelim}

\subsection{Basic definitions and examples} Given a subset $S \subseteq B(\mc{H})$ of bounded operators on a Hilbert space $\mc{H}$, we define the \textbf{commutant} of $S$, denoted $S'$, by \[S' :=\left\{a \in B(\mc{H})\  | \ sa = as \text{ for every } s \in S\right\}.\]  A \textbf{von Neumann algebra} is a unital, self-adjoint\footnote{$x \in M \Leftrightarrow x^* \in M$} subalgebra $M$ of $B(\mc{H})$ for some Hilbert space $\mc{H}$ with the property that $M'' := (M')'$ (the \textbf{bicommutant} of $M$) is equal to $M$.

Recall that a \textbf{tracial von Neumann algebra} is given by a pair $(M,\tau)$, where $M$ is a von Neumann algebra and $\tau$ is a faithful normal tracial state on $M$.   Given two tracial von Neumann algebras $(N,\tau)$ and $(M,\sigma)$, an \textbf{embedding} of $(N,\tau)$ into $(M,\sigma)$ is an injective unital $*$-homomorphism $\pi: (N,\tau) \hookrightarrow (M, \sigma)$ such that $\sigma \circ \pi = \tau$.  When context is clear, we drop the traces and just write $\pi: N\hookrightarrow M$.  Given an embedding $\pi: N \hookrightarrow M$, we will often consider its \textbf{relative commutant} $\pi(N)'\cap M$. Two embeddings $\pi_1,\pi_2: N \hookrightarrow M$ are \textbf{unitarily equivalent} if there exists a unitary $u \in M$ such that, for every $x \in N, \pi_1(x) = u^*\pi_2(x) u$. Two embeddings $\pi_1,\pi_2: N \hookrightarrow M$ are \textbf{automorphically equivalent} if there exists an automorphism $\alpha \in \text{Aut}(M)$ such that $\pi_1 = \alpha \circ \pi_2$.

Due to the general result that any von Neumann algebra can be realized as a direct integral of factors over its center, the study of tracial von Neumann algebras is often reduced to the study of so-called \tql II$_1$ factors.\tqr  A \textbf{II$_1$ factor} is an infinite-dimensional tracial von Neumann algebra $(M,\tau)$ that is also a \textbf{factor}: the center of $M$, denoted $\mc{Z}(M)$, is trivial, that is, $\mc{Z}(M) = \mb{C}$.  The property of $M$ being a factor is commonly expressed by the equality $M' \cap M = \mb{C}$.  

The following fact provides two useful characterizations of II$_1$ factors:

\begin{fact}\label{sametrace}
Let $(M,\tau)$ be an infinite-dimensional tracial von Neumann algebra.  The following are equivalent:

\begin{enumerate}
    \item $M$ is a II$_1$ factor;
    \item $M$ has a unique faithful normal tracial state;
    \item For any pair of projections $p,q \in M$ with $\tau(p) = \tau(q)$, there is a unitary $u \in M$ such that $p = u^*qu$.
\end{enumerate}  
\end{fact}

Given a II$_1$ factor $M$\footnote{or, more generally, a tracial von Neumann algebra $(M,\tau)$} with faithful normal tracial state $\tau$, a fundamental tool in the analysis of $M$ is the \textbf{Hilbert-Schmidt norm} on $M$, denoted $||\cdot||_{2,\tau}$, defined by \[||x||_{2,\tau} = \sqrt{\tau(x^*x)}, \quad x\in M.\]  When context is clear, we will suppress the $\tau$ in the subscript and simply write $||\cdot ||_2$.  The Hilbert-Schmidt norm induces a pre-Hilbert space structure on a II$_1$ factor with inner product given by $\left\langle x | y \right\rangle := \tau(y^*x)$. We denote the completion of $M$ under the Hilbert-Schmidt norm by $L^2(M,\tau)$ (or simply $L^2(M)$). A II$_1$ factor is \textbf{separably acting} if it can be faithfully represented on a separable Hilbert space.  Evidently, a II$_1$ factor $M$ is separably acting if and only if $L^2(M)$ is separable. It is a common abuse of terminology--one which we willingly commit in this article--to call a separably acting II$_1$ factor \textbf{separable}.

Given two elements $x, y$ in a von Neumann algebra, we will often have reason to consider their \textbf{commutator}, denoted $[x,y]$, given by \[[x,y] = xy - yx.\]

\begin{example}\label{Rexample}
  The most well-known example of a II$_1$ factor is the \tql separably acting hyperfinite II$_1$ factor,\tqr denoted by $\R$.  Murray and von Neumann showed in \cite{mvn4} that $\R$ is the unique separable hyperfinite II$_1$ factor.  To sketch a construction, consider the infinite tensor product $\ds \bigotimes_\mb{N} \mb{M}_2$.  Using the unique tracial state, form a GNS representation\footnote{Gelfand-Naimark-Segal representation--see \cite{davidson}.} of $\ds \bigotimes_\mb{N} \mb{M}_2$ and take the bicommutant.  By the uniqueness of $\R$, one could also construct $\R$ by considering a tensor product $\ds \bigotimes_\mb{N} \mb{M}_{k(n)}$ where $\left\{k(n)\right\}$ is any sequence of natural numbers with $k(n) \geq 2$ and taking the bicommutant.  In addition to this construction, there are several other ways to realize $\R$.
\end{example}

Next, we turn to address the term \tql hyperfinite\tqr appearing in the previous example.  Consider the following two definitions:

\begin{defn}\label{amendef}
 \
\begin{enumerate}
\item  A von Neumann algebra $M$ is \textbf{hyperfinite} if it can be expressed as the $\sigma$-weak closure of an increasing union of finite-dimensional subalgebras.

\item A von Neumann algebra $M$ is \textbf{injective} if for any inclusion $X \subseteq Y$ of operator systems and ucp\footnote{For every $n \in \mb{N}$, any amplification $\varphi^{(n)}: \mb{M}_n(X) \rightarrow \mb{M}_n(M)$ given by $\varphi^{(n)}((x_{ij})) = (\varphi(x_{ij}))$ is positive, that is, positive elements are sent to positive elements.  See \cite{paulsen}.} map $\varphi: X \rightarrow M$, there exists a ucp map $\tilde{\varphi}: Y \rightarrow M$ such that $\tilde{\varphi}|_X = \varphi$.

\end{enumerate}
\end{defn}

 It is well-known that the above are equivalent.  This equivalence is due to the groundbreaking result from \cite{connes}, but we should also mention \cite{choeff,wass,newproof} when discussing this result.  There are more conditions well-known to be equivalent to hyperfiniteness, such as amenability and semidiscreteness, but since they make no appearance in this paper, we refrain from defining them.
 
 Another class of II$_1$ factors relevant to this article is the class of \tql McDuff\tqr\  II$_1$ factors.  Such factors were first studied in McDuff's revolutionary article \cite{mcduff}.  A II$_1$ factor $M$ is \textbf{McDuff} if $M \cong M \otimes \R$.\footnote{This is actually a theorem appearing in \cite{mcduff}; the original definition is that $M$ is McDuff if it has a pair of central sequences that do not commute with each other.} From the construction in Example \ref{Rexample}, it can be deduced that $\R$ is McDuff.
 
\subsection{Ultraproducts of tracial von Neumann algebras}\label{ultraprelim}

In this subsection, we will discuss the ultraproduct construction for tracial von Neumann algebras. Let $\mc{U}$ denote a nonprincipal ultrafilter\footnote{See Appendix A of \cite{brownozawa} for an operator algebraist-friendly introduction to ultrafilters} on $\mb{N}$.  For each $k \in \mb{N}$, let $(M_k,\tau_k)$ be a tracial von Neumann algebra, and let $||\cdot ||_{2,k}$ denote the induced Hilbert-Schmidt norm on $M_k$. Consider the sequence space \[\prod_{k \in \mb{N}}^\infty M_k:= \left\{(x_k)_{k \in \mb{N}}: x_k \in M_k \text{ and } \sup_k ||x_k|| < \infty\right\}.\]  We define the \textbf{tracial ultraproduct} of the $(M_k,\tau_k)$'s (with respect to $\u$), denoted $\ds \prod_{k\rightarrow \mc{U}}(M_k,\tau_k)$ (or simply $\ds\prod_{k\rightarrow \mc{U}}M_k$ when the context is clear), to be given by \[\prod_{k\rightarrow \mc{U}}(M_k,\tau_k) := \left(\prod_{k \in \mb{N}}^\infty M_k \right)\Big/ \mc{I}_\mc{U},\] where \[\ds \mc{I}_\mc{U} := \left\{ (x_k) \in \prod_{k \in \mb{N}}^\infty M_k : \lim_{k \rightarrow \mc{U}} ||x_k||_{2,k} = 0\right\}.\] Given a sequence $\ds(x_k) \in \prod_{k \in\mb{N}}^\infty M_k$, let $(x_k)_\mc{U}$ denote the coset of $(x_k)$ in $\ds\prod_{k\rightarrow \mc{U}}(M_k,\tau_k)$.  The ultraproduct $\ds \prod_{k\rightarrow \u} M_k$ has a natural tracial state $\tau_\u$ given by $\tau_\u((x_k)_\u) = \lim_{k\rightarrow \u} \tau_k(x_k)$.  This tracial state induces the Hilbert-Schmidt norm $||\cdot||_{2,\tau_\u}$ on $\ds \prod_{k\rightarrow \u} M_k$.

We now take this opportunity to record some facts about (tracial)\footnote{We will exclusively consider tracial ultraproducts in this article, so in the sequel we will drop the \tql tracial\tqr modifier.} ultrapowers.

\begin{fact}
If $M_k$ is a II$_1$ factor for every $k \in \mb{N}$, then $\ds \prod_{k\rightarrow \u} M_k$ is also a II$_1$ factor.
\end{fact}

\begin{fact}\label{stable}
For each $k \in \mb{N}$, let $M_k$ be a II$_1$ factor.
\begin{enumerate}
    \item If $\ds u \in \prod_{k\rightarrow \u} M_k$ is a unitary, then there exist unitaries $u_k \in M_k$ for every $k \in \mb{N}$ such that $u = (u_k)_\u$.
    \item If $\ds p \in \prod_{k\rightarrow \u} M_k$ is a projection, then there exis projections $p_k \in M_k$ for every $k \in \mb{N}$ such that $p = (p_k)_\u$.  Furthermore, each $p_k$ can be chosen such that $\tau_k(p_k) = \tau_\u(p)$.
\end{enumerate}

\end{fact}

If $(M_k,\tau_k) = (M,\tau)$ for every $k \in \mb{N}$, we write $\ell^\infty(M)$ for $\ds\prod_{k \in \mb{N}}^\infty M_k$ and $M^\u$ for $\ds \prod_{k\rightarrow \u} M_k$, and we call $M^\u$ the \textbf{ultrapower} of $M$ (with respect to $\u$).  This article will mostly address ultrapowers.  It is important to note that there is always a canonical embedding of $M$ into its ultrapower $M^\u$ given by the \textbf{diagonal} (or \textbf{constant sequence}) \textbf{embedding} $x \mapsto (x)_\mc{U}$ (the coset of the constant sequence with $x$ in every entry).  We sometimes abuse notation and write $M\subset M^\u$ by identifying $M$ with its image under the diagonal embedding.

Whether or not the isomorphism type of the ultrapower depends on the choice of ultrafilter is sensitive to set theory.  More specifically, given a separable II$_1$ factor, all of its ultrapowers with respect to nonprincipal ultrafilters on $\mathbb N$ are isomorphic if and only if CH holds.  (See \cite{mtoa1}.)  That being said, since we are always working under the assumption that CH holds, in this paper, we make the following convention:

\begin{conv}
Throughout this paper, $\u$ denotes a fixed nonprincipal ultrafilter on $\mathbb N.$
\end{conv}

A benefit of considering ultrapowers of II$_1$ factors is that the ultrapower setting provides a formal way to concisely express many approximation properties.  For example, a II$_1$ factor has \textbf{Property Gamma} if $M'\cap M^\u \neq \mb{C}$ and a II$_1$ factor is McDuff if and only if $M'\cap M^\u$ is nonabelian.

\subsection{Survey of  $\HOM(N,M^\mc{U})$}\label{convex}

As mentioned in the introduction, the space \linebreak $\HOM(N,M^\u)$ plays a significant role in the proof of our main results.  This space was first studied by Brown in \cite{topdyn} in the case that $M=\R$. Let $N, P$ be II$_1$ factors.  Let $\HOM(N,P)$ denote the space of all embeddings $\pi: N \hookrightarrow P$ modulo unitary equivalence.  Given an embedding $\pi: N\hookrightarrow P$, denote its unitary equivalence class by $[\pi]$. We can endow this space with a topology best described as \tql point-$||\cdot||_2$ convergence along representatives:'' $[\pi_n]\rightarrow [\pi]$ in $\HOM(N,P)$ if there exist representatives $\pi_n' \in [\pi_n]$ such that, for every $x \in N$, $||\pi_n'(x) - \pi(x)||_2 \rightarrow 0$.  In \cite{topdyn}, Brown considered the space $\HOM(N,\R^\u)$ where $N$ is a separably acting embeddable II$_1$ factor.  One of the main results of \cite{topdyn} was that $\HOM(N,\R^\u)$ satisfies the axioms for a convex-like structure.\footnote{With no ambient linear space containing $\HOM(N,\R^\u)$, Brown defined axioms that every convex space should satisfy.  It was subsequently shown in \cite{capfri} that a space satisfying these axioms can be realized as a convex subset of a Banach space.}  In \cite{brocap} and \cite{saa}, this convex structure was extended to the more general setting of $\HOM(N,M^\u)$, where $M$ is a McDuff II$_1$ factor.  

We now define convex combinations in $\HOM(N,M^\u)$ for $M$ a McDuff II$_1$ factor.  Let $\sigma: M\otimes \R \rightarrow M$ be an isomorphism such that the map $x\rightarrow \sigma(x\otimes 1_\R)$ is weakly approximately unitarily equivalent to $\text{id}_M$.\footnote{This means that there is a sequence of unitaries $\left\{u_n\right\}$ in $M$ such that, for every $x \in N, ||u_n^*\sigma(x\otimes 1_\R)u_n - x||_2 \rightarrow 0$.} 

\begin{defn}[\cite{topdyn,saa}]

Given $[\pi],[\rho] \in \HOM(N,M^\u)$ and $t \in [0,1]$ we put \[t[\pi] + (1-t)[\rho] := [\sigma(\pi\otimes p) + \sigma(\rho \otimes p^\perp)],\] where $p$ is a projection in $\R^\u$ with trace $t$, $p^\perp = 1_{\R^\u} - p$, and $\sigma(\pi \otimes p)$ is the map given by $x\mapsto \sigma(\pi(x)\otimes p)$ (likewise for $\sigma(\rho\otimes p^\perp)$).
\end{defn}

This operation is well-defined and satisfies the axioms for a convex-like structure.  In \cite{topdyn}, Brown established a characterization of extreme points in the convex structure $\HOM(N,\R^\u)$ which was later extended to $\HOM(N,M^\u)$, where $M$ is a McDuff II$_1$ factor, in \cite{saa} and can be stated as follows:

\begin{thm}[\cite{topdyn,saa}]\label{exptchar}
Let $M$ be a McDuff II$_1$ factor. An equivalence class $[\pi] \in \HOM(N,M^\u)$ is extreme if and only if the relative commutant $\pi(N)'\cap M^\u$ is a factor.
\end{thm}

This says that embeddings of $N$ into $M^\u$ with factorial relative commutant are the irreducible objects in this context. 

The following terminology will prove useful throughout this paper:

\begin{defn}
Suppose that $N$ embeds into $M^\u$.  We say that $(N,M)$ is a:
\begin{enumerate}
    \item \textbf{factorial commutant pair} if there is an embedding $\pi:N\hookrightarrow M^\u$ such that $\pi(N)'\cap M^\u$ is a factor;
    \item \textbf{strong factorial commutant pair} if every embedding $\pi:N\hookrightarrow M^\u$ is such that $\pi(N)'\cap M^\u$ is a factor.
\end{enumerate}
\end{defn}

Theorem \ref{exptchar} yields the following characterization of $\R$:

\begin{cor}(\cite[Corollary 5.3]{topdyn})\label{Relcom}
For any separable embeddable II$_1$ factor $N$, we have that $(N,\R)$ is a strong factorial commutant pair if and only if $N\cong \R$.
\end{cor}
This was later strengthened in Theorem 5.8 of \cite{saa}.

Recall Popa's question from the introduction: does every separable embeddable II$_1$ factor admit an embedding into $\R^\u$ with factorial relative commutant, or, in our current terminology, if $N$ is a separable embeddable II$_1$ factor, is $(N,\R)$ always a factorial commutant pair?  The above characterization of extreme points provides a convex-geometric interpretation of Popa's question: for any separable embeddable II$_1$ factor $N$, does $\HOM(N,\R^\u)$ have an extreme point?  This question remains open in general.  Brown made some progress on this problem in \cite{topdyn}. Indeed, the following result due to Brown in regards to this question on existence of extreme points is crucial to the results of this article:

\begin{thm}(\cite[Theorem 6.9]{topdyn})\label{nate}
For any separable $M \subset \R^\u$, there is a separable II$_1$ factor $N \subset \R^\u$ such that $M\subset N$ and $N' \cap \R^\u$ is a factor.\footnote{In fact, one can take $N = M*L(SL_3(\mb{Z}))$.}
\end{thm}


The reader interested in seeing more details and results on $\HOM(N,M^\u)$ is directed to \cite{topdyn,brocap,saa,fidisimp}.


\section{Model-theoretic preliminaries}\label{mtprelim}

In this section, we give a brief survey of some of the fundamental notions of continuous model theory as they apply to tracial von Neumann algebras.  One can consult \cite{mtfms}, \cite{MTOA2}, or \cite{munster} for more detailed explanations. 

\subsection{Basic model-theoretic notions}

We treat tracial von Neumann algebras as model-theoretic structures using an appropriate continuous first-order logic.  We start with \textbf{atomic formulae} $\varphi(x)$ (here $x$ is a tuple of variables), which are simply expressions of the form $\tr(p(x))$ for some $*$-polynomial $p(x)$.\footnote{Technically, since our logic is ``real-valued,'' we have two such expressions, one for the real part of the trace and one for the imaginary part.}  We obtain the class of all \textbf{formulae} by closing the atomic formulae under applications of continuous functions $\mathbb R^n\to \mathbb R$ (as $n$ varies over $\mathbb N$) and the ``quantifiers'' $\sup_x$ and $\inf_x$ (where the variables range over operator-norm bounded balls).  

\begin{example}\label{formulaexample}
Consider the formula $\varphi(x)$ that is $\sup_y (\|[x,y]\|_2\dotminus \epsilon)$.  
The function $r \dotminus \epsilon$ is defined to be $\max(r-\epsilon,0)$ (which is clearly continuous).  For simplicity, we have omitted what operator norm ball $y$ is ranging over, but we usually assume our quantifiers range over the unit ball (which is often enough).
\end{example}

Given a formula $\varphi(x)$, a tracial von Neumann algebra $M$, and a tuple $a\in M$, there is the notion of the \textbf{interpretation} $\varphi(a)^M$, which is simply what one gets when plugging $a$ in for the free variables $x$ and evaluating.  For example, with $\varphi(x)$ as in Example \ref{formulaexample}, $\varphi(a)^M=0$ if and only if $\|[a,b]\|_2\leq \epsilon$ for all $b$ in the unit ball of $M$.

Given a formula $\varphi(x,y)$, a tracial von Neumann algebra $M$, and a tuple $b$ from $M$, we also consider the expression $\varphi(x,b)$, where we replace all occurrences of the variables $y$ with the tuple $b$.  We refer to such an expression as a formula with \textbf{parameters} $b$.

A \textbf{sentence} is a formula with no free variables.  For example, we could consider the formula $\varphi(x)$ from Example \ref{formulaexample} and form the sentence $\psi:=\inf_x\varphi(x)$.  Note then that, given a tracial von Neumann algebra $M$, we have that $\psi^M=0$ if and only if, for any $\delta>\epsilon$, there is $a$ in the unit ball of $M$ such that $\|[a,b]\|_2<\delta$ for all $b$ in the unit ball of $M$.

Tracial von Neumann algebras $M$ and $N$ are said to be \textbf{elementary equivalent}, denoted $M\equiv N$, if, for any sentence $\psi$, one has $\psi^M=\psi^N$.  This is the so-called \textbf{syntactic characterization} of elementary equivalence.  One can give an alternative, \textbf{semantic} definition, which is often more appealing to operator algebraists, namely separable tracial von Neumann algebras $M$ and $N$ are elementarily equivalent if $M^\u\cong N^\u$.\footnote{This heavily uses our standing CH assumption.  The \textbf{Keisler-Shelah Theorem} provides a similar characterization that does not depend on set theory nor the fact that $M$ and $N$ are separable; see \cite{hensoniovino}.}  

Elementary equivalence is a much coarser equivalence relation than isomorphism.  In fact, in \cite[Theorem 4.3]{MTOA3}, Farah, Hart, and Sherman proved the following:

\begin{fact}
For any separable II$_1$ factor $M$, there are continuum many non-isomorphic separable II$_1$ factors $N$ such that $M\equiv N$.
\end{fact}

If $M$ and $N$ are tracial von Neumann algebras, then an embedding $j:N\hookrightarrow M$ is said to be \textbf{elementary} if, for any formula $\varphi(x)$ and tuple $a\in N$, one has $\varphi(a)^N=\varphi(j(a))^M$.  This also can be given a semantic reformulation:  $j:N\hookrightarrow M$ is an elementary embedding if and only if it can be extended to an isomorphism $N^\u\cong M^\u$.\footnote{Again, this uses our CH assumption.}  In particular, if there is an elementary embedding $N\hookrightarrow M$, then $N\equiv M$.  

\begin{fact}[Elementary facts about elementary embeddings]\label{elel}

\

\begin{enumerate}
\item Isomorphisms are elementary embeddings.
\item Suppose that $i:M\hookrightarrow N$ and $j:N\hookrightarrow P$ are embeddings.  If $i$ and $j$ are both elementary, then so is $ji$.  If $j$ and $ji$ are both elementary, then so is $i$. 
\item If one has a directed system of tracial von Neumann algebras with each embedding elementary, then the canonical embeddings into the direct limit are also elementary.
\end{enumerate}
\end{fact}

In case the directed system from item (3) of the previous lemma is linearly ordered, we often refer to the corresponding directed system as an \textbf{elementary chain} of tracial von Neumann algebras.

\noindent If $N$ is a subalgebra of $M$, then $N$ is said to be an \textbf{elementary} subalgebra of $M$, denoted $N\preceq M$, if the inclusion map $N\hookrightarrow M$ is elementary.  Of fundamental importance is the following:

\begin{fact}[Downward L\"owenhim-Skolem]\label{DLS}
If $M$ is a tracial von Neumann algebra and $X\subseteq M$ an arbitrary subset, then there is $N\preceq M$ with $X\subseteq N$.  Moreover, one can take $N$ to have the same density character\footnote{Here, the density character of a subset of a tracial von Neumann algebra is the cardinality of the smallest dense subset of that set.} as $X$.
\end{fact}

Given a tracial von Neumann algebra $M$ with subsets $A$ and $B$, a map $j:A\to B$ is said to be \textbf{partial elementary} if $\varphi(a)^M=\varphi(j(a))^M$ for all formulae $\varphi(x)$ and all tuples $a\in A$.  Clearly, a partial elementary map is an isometric embedding.

The following theorem explains one of the main reasons that ultrapowers are of fundamental importance in model theory (see \cite[Theorem 5.4]{mtfms}):

\begin{fact}[\L os' theorem]
For any formula $\varphi(x)$ and any tuple $a=(a_k)_\mc{U}$ from $M^\u$, we have 
$$\varphi(a)^{M^\u}=\lim_{k\to\u} \varphi(a_k)^M.$$ In particular, the diagonal embedding $M\hookrightarrow M^\u$ is elementary.
\end{fact}

The following are immediate consequences of \L os' theorem:

\begin{fact}\label{ultrauniv}

\

\begin{enumerate}
\item Ultrapowers of elementary embeddings are elementary:  if $j:N\hookrightarrow M$ is elementary, then the natural ultrapower map $j^\u:N^\u\hookrightarrow M^\u$ is also elementary.
\item (Separable universality of ultrapowers) If $N\equiv M$ and $N$ is separable, then there is an elementary embedding $N\hookrightarrow M^\u$.
\end{enumerate}
\end{fact}

We will also need the following facts about elementary embeddings particular to $\R$ and its ultrapower.

\begin{fact}\label{Relementary}

\

\begin{enumerate}
\item Every embedding of $\R$ into $\R^\u$ is elementary.
\item Suppose that $M\equiv \R$.  Then every embedding $\R\hookrightarrow M$ is elementary.
\end{enumerate}
\end{fact}

\begin{proof}
(1) follows from \L os' theorem and the fact that any two embeddings of $\R$ into its ultrapower are unitarily equivalent.  (2) follows from the first item, separable universality of $\R^\u$ (Fact \ref{ultrauniv}(2)), and Fact \ref{elel}(2).
\end{proof}

Another key property of ultrapowers is that they are somewhat \textbf{saturated}:

\begin{fact}[Separable saturation of ultrapowers]
Fix a separable set $A\subseteq M^\u$ and a collection $(\varphi_i(x,a_i))_{i\in I}$ of formulae with parameters from $A$.  Then the following are equivalent:
\begin{enumerate}
\item{(approximate finite satisfiability)} For any finite $I_0\subseteq I$ and any $\epsilon>0$, there is $a\in M^\u$ such that $\varphi(a,a_i)^{M^\u}<\epsilon$ for all $i\in I_0$;
\item {(satisfiability)} There is $a\in M^\u$ such that $\varphi_i(a,a_i)^{M^\u}=0$ for all $i\in I$.
\end{enumerate}
\end{fact}

Finally, we will need the following separable homogeneity property of ultrapowers.  The proof is a standard ``back and forth'' argument using CH and separable saturation.

\begin{fact}\label{ultrahomog}
Suppose that $A$ and $B$ are separable subsets of $M^\u$ and $j:A\to B$ is a surjective partial elementary map.  Then there is an automorphism $\alpha$ of $M^\u$ extending $j$.
\end{fact}

\subsection{Types}\label{types}

\begin{defn}
Given a separable subset $A\subseteq M^\u$ and a tuple $a\in M^\u$, we define the \textbf{type of $a$ in $M^\u$ over $A$}, denoted $\tp^{M^\u}(a/A)$ (or simply $\tp(a/A)$ if the ambient ultrapower is clear from context), to be the function which assigns to every formula $\varphi(x,b)$ with $b$ a tuple of parameters from $A$ the value $\varphi(a,b)^{M^\u}$.  A \textbf{type in $M^\u$ over $A$} is a function of the form $\tp(a/A)$ for some $a\in M^\u$.
\end{defn}

Thus, $\tp(a/A)$ is a description of every first-order fact about $a$ we might want to know using parameters from $A$.

For a separable subset $A$ of $M^\u$, we let $S(A)$ denote the set of 1-types over $A$, that is, the set of types of single elements in $M^\u$ over $A$.  We often use $p$ and $q$ to denote types.  We write $\varphi(x,b)^p$ for the value of the function $p$ on the formula $\varphi(x,b)$.  In other words, if $a\in M^\u$ \textbf{realizes} $p$, meaning that $p=\tp(a/A)$, then $\varphi(x,b)^p=\varphi(a,b)^{M^\u}$.  We also let $p(M^\u)$ denote the set of realizations of $p$ in $M^\u$.

The nontrivial direction of the next fact follows immediately from separable homogeneity of ultrapowers.

\begin{fact}\label{sametype}
If $A$ is a separable subset of $M^\u$ and $a,b\in M^\u$ are two tuples of the same length, then $\tp(a/A)=\tp(b/A)$ if and only if there is an automorphism $\alpha$ of $M^\u$ that fixes $A$ pointwise and such that $\alpha(a)=b$.
\end{fact}

The previous fact shows that one may alternatively view elements of $S(A)$ as orbits in $M^\u$ under the natural action of $\operatorname{Aut}(M^\u/A)$, the group of automorphisms of $M^\u$ that fix $A$ pointwise.  Here is an example to show how this perspective can be useful:

\begin{example}\label{typecommutant}
  Suppose that $N$ is a separable subalgebra of $M^\u$ and $a\in N'\cap M^\u$.  Setting $p:=\tp(a/N)$, one then has that $p(M^\u)\subseteq N'\cap M^\u$.  If, in addition, $a\in Z(N'\cap M^\u)$, then $p(M^\u)\subseteq Z(N'\cap M^\u)$.
\end{example}

If $A\subseteq B$ are separable subsets of $M^\u$ and $p\in S(A)$ and $q\in S(B)$, then we write $p\subseteq q$ if $q$ extends $p$ as a function, that is, for every formula $\varphi(x,b)$ with parameters from $A$, we have $\varphi(x,b)^p=\varphi(x,b)^q$.  We refer to $q$ as an extension of $p$ to $B$ and $p$ as the restriction of $q$ to $A$.  Note that from the orbit perspective, if $p\subseteq q$, then the $\operatorname{Aut}(M^\u/B)$-orbit corresponding to $q$ is contained in the $\operatorname{Aut}(M^\u/A)$-orbit corresponding to $p$.

Crucial to our proof of the main theorem of this paper is the existence of a special kind of extension of types called heirs.

\begin{defn}
Suppose that $N$ and $A$ are separable subsets of $M^\u$ with $N\preceq M^\u$ and $N\subseteq A\subseteq M^\u$.  If $p\in S(N)$ and $q\in S(A)$ are such that $p\subseteq q$, we say that $q$ is an \textbf{heir} of $p$ if, for every formula $\varphi(x,b)$ with parameters from $A$ and every $\epsilon>0$, there is $c\in N$ such that
$$|\varphi(x,b)^q-\varphi(x,c)^p|<\epsilon.$$
\end{defn}

The notion of an heir might appear technical at first glance so we offer the following heuristic explanation.  The type $p$ as in the definition gathers all first-order information about some element (a realization of the type) using parameters from $N$.  The extension $q$ is now adding to this information by also describing how the realization should interact with parameters from the larger set $A$.  $q$ is then an heir of $p$ if no ``new phenomena'' occur in $q$, that is, if a first-order phenomena occurs in $q$, then it also (approximately) occurs in $p$.  The next example explains exactly how heirs will be used in the next section:

\begin{example}\label{heircommutant}
  Suppose that $P\subseteq Q\subseteq M^\u$ are separable subalgebras of $M^\u$, $a\in P'\cap M^\u$, $p:=\tp(a/P)$, and $q$ is an heir of $p$ to $Q$.  Then $q(M^\u)\subseteq Q'\cap M^\u$.  To see this, let $\varphi(x,y)$ be the formula $\|[x,y]\|_2$.  Since $\varphi(x,b)^p=0$ for every $b\in P$, we must have that $\varphi(x,b)^q=0$ for every $b\in Q$.  Indeed, if this were not the case, that is, if $\varphi(x,b)^q=r>0$ for some $b\in Q$, then there would be $c\in P$ such that $|\varphi(x,b)-\varphi(x,c)|^q<\frac{r}{2}$ by the heir property, whence $\varphi(x,c)^p>\frac{r}{2}$, which is a contradiction.
\end{example}

The following fact is standard in the classical setting; the only mention of it in the continuous setting is \cite{stablegroups}, where it is mentioned to follow from a ``compactness argument.''  For the sake of the reader, we provide this argument.

\begin{fact}\label{heir}
For any separable subsets $N$ and $A$ of $M^\u$ with $N\preceq M^\u$ and $N\subseteq A\subseteq M^\u$, and any $p\in S(N)$, there is $q\in S(A)$ that is an heir of $p$.
\end{fact}

\begin{proof}
We seek $a\in M^\u$ satisfying the following two kinds of conditions:
\begin{enumerate}
\item $\psi(a)=\psi(x)^p$ for any formula $\psi(x)$ with parameters from $N$;
\item $\varphi(a,c)^{M^\u}\geq \frac{\epsilon}{2}$ for any formula $\varphi(x,c)$ with parameters from $A$ and any $\epsilon>0$ such that $\varphi(x,b)^p\geq \epsilon$ for all $b\in N$.
\end{enumerate}
Indeed, if $a$ is as above, we claim that $q:=\tp(a/A)$ is an heir of $p$.  By (1), $q$ is an extension of $p$.  To see that $q$ is an heir, fix a formula $\varphi(x,c)$ with parameters from $A$ and set $s:=\varphi(x,c)^q=\varphi(a,c)^{M^\u}$.  Suppose, towards a contradiction, that there is $\epsilon>0$ such that $|\varphi(x,b)^p-s|\geq \epsilon$ for all $b\in N$.  It follows that $|\varphi(x,b)-s|^p\geq \epsilon$ for all $b\in N$, whence, by (2), $|\varphi(a,c)^{M^\u}-s|\geq \frac{\epsilon}{2}$, leading to a contradiction.

Suppose now, towards a contradiction, that no such $a\in M^\u$ exists.  By separable saturation, it follows that there are:
\begin{itemize}
\item a formula $\psi(x)$ with parameters from $N$ such that $\psi(x)^p=0$, 
\item a $\delta>0$, and 
\item formulae $\varphi_1(x,c_1),\ldots,\varphi_m(x,c_n)$ with parameters from $A$ as in (2)
\end{itemize}
such that, for any $a\in M^\u$, if $\psi(a)<\delta$, then $\varphi_i(a,c_i)<\frac{\epsilon}{2}$ for some $i=1,\ldots,m$. 

In other words, 
$$\left(\sup_x\min\left(\delta\dotminus \psi(x),\min_{1\leq i\leq m}\left(\varphi_i(x,c_i)\dotminus \frac{\epsilon}{2}\right)\right)\right)^{M^\u}=0.$$  Consequently, $$\left(\inf_{y_1}\cdots\inf_{y_m}\sup_x\min\left(\delta\dotminus \psi(x),\min_{1\leq i\leq m}\left(\varphi_i(x,y_i)\dotminus \frac{\epsilon}{2}\right)\right)\right)^{M^\u}=0,$$ and thus 
$$\left(\inf_{y_1}\cdots\inf_{y_m}\sup_x\min\left(\delta\dotminus \psi(x),\min_{1\leq i\leq m}\left(\varphi_i(x,y_i)\dotminus \frac{\epsilon}{2}\right)\right)\right)^M=0.$$ Set $\eta:=\min(\delta,\frac{\epsilon}{2})$ and take $d_1,\ldots,d_m\in M$ such that $$\left(\sup_x\min\left(\delta\dotminus \psi(x),\min_{1\leq i\leq m}\left(\varphi_i(x,d_i)\dotminus \frac{\epsilon}{2}\right)\right)\right)^M<\eta,$$ whence $$\left(\sup_x\min\left(\delta\dotminus \psi(x),\min_{1\leq i\leq m}\left(\varphi_i(x,c_i)\dotminus \frac{\epsilon}{2}\right)\right)\right)^{M^\u}<\eta.$$  Take $a\in M^\u$ realizing $p$.  Then $\psi(a)^{M^\u}=\psi(x)^p=0$, whence, since $\eta\leq \delta$, we have $\min_{1\leq i\leq m}(\varphi_i(x,c_i)\dotminus \frac{\epsilon}{2})^{M^\u}<\eta\leq \frac{\epsilon}{2}$.  Choosing $i$ such that $(\varphi_i(a,d_i)\dotminus \frac{\epsilon}{2})^{M^\u}<\eta$, we get that $\varphi_i(x,d_i)^p=\varphi_i(a,d_i)^{M^\u}<\epsilon$, a contradiction.
\end{proof}

\subsection{Existentially closed factors}

The following notion is the model-theoretic generalization of the notion of algebraically closed field.  It has been extensively studied in the operator algebraic context (see \cite{ECfactors} and \cite{KEP}).

\begin{defn}
Suppose that $M$ is a subalgebra of the separable tracial von Neumann algebra $N$.  We say that $M$ is \textbf{existentially closed (e.c.) in $N$} if there is an embedding $j:N\hookrightarrow M^\u$ such that the restriction of $j$ to $M$ is the diagonal embedding $M\hookrightarrow M^\u$.  We say that a separable tracial von Neumann algebra $M$ is \textbf{existentially closed (e.c.)}\footnote{We are giving the semantic definition here and are making use of our standing CH assumption.  The syntactic definition states that an ``existential'' sentence with parameters from $M$ has the same value in $M$ as it does in any extension.  The syntactic definition also does not have any separability requirements.} if $M$ is e.c. in $N$ whenever $N$ is a separable tracial von Neumann algebra containing $M$.
\end{defn}

Items (1)-(3) of the following can be found in \cite{nomodcomp} and \cite{ECfactors}; item (4) follows immediately from the definition.

\begin{fact}\label{ecfacts}

\

\begin{enumerate}
    \item E.c. tracial von Neumann algebras are McDuff II$_1$ factors.
    \item Every separable tracial von Neumann algebra embeds into an e.c. factor.
    \item E.c. factors are \textbf{locally universal}, that is, if $M$ is an e.c. factor, then any separable tracial von Neumann algebra  embeds into $M^\u$. 
    \item If $M_0\subseteq M_1\subseteq M_2\subseteq \cdots$ is a chain of e.c. factors with union $M$, then $M$ is also e.c.
\end{enumerate}
\end{fact}

We will also need to consider a relative version of this notion where we restrict to $P^\u$-embeddable algebras for some tracial von Neumann algebra $P$:

\begin{defn}
If $M$ and $P$ are separable tracial von Neumann algebras such that $M$ embeds into $P^\u$, then $M$ is an \textbf{existentially closed (e.c.) $P^\u$-embeddable} algebra if $M$ is e.c. in $N$ whenever $N$ is a separable $P^\u$-embeddable tracial von Neumann algebra containing $M$.  When $P=\R$, we simply say that $M$ is an existentially closed embeddable algebra.
\end{defn}

Once again, every separable $P^\u$-embeddable algebra embeds into a separable e.c. $P^\u$-embeddable algebra.  The same proof that e.c. tracial von Neumann algebras are locally universal shows that if $M$ is an e.c. $P^\u$-embeddable algebra, then $M$ and $P$ are \textbf{mutually embeddable}, that is, $P$ is also $M^\u$-embeddable.  Finally, if $P$ is a factor, then any e.c. $P^\u$-embeddable algebra is also a factor\footnote{To see this, it suffices to show that any $P^\u$-embeddable algebra embeds into a $P^\u$-embeddable factor; this follows, for example, from \cite[Corollary 0.2]{popa}.}; in particular, e.c. embeddable algebras are factors.


Although we will not need it in this paper, one should observe that being an e.c. (embeddable) factor is not an \emph{axiomatizable} property in that it is not preserved under ultraproducts.  This fact was first observed in \cite{nomodcomp} for arbitrary II$_1$ factors and then in \cite{ECfactors} for embeddable II$_1$ factors.  Since it relates to the work of the first- and third-named authors mentioned above, we offer a different argument for this latter fact.   

\begin{thm}
$\R^\u$ is not an e.c. embeddable factor.
\end{thm}

\begin{proof}
Take $N\equiv \R$ such that such that $N\not\cong \R$.  Fix elementary embeddings $\R\hookrightarrow N$ and $N\hookrightarrow \R^\u$; the first exists by Fact \ref{Relementary}(2) and the second exists by Fact \ref{ultrauniv}(2).  Since the composite embedding $j:N\hookrightarrow N^\u$ is elementary, there is an automorphism $\alpha$ of $N^\u$ that passes $j$ to the diagonal embedding; this follows from Fact \ref{ultrahomog}.

If $N^\u$ were an e.c. embeddable factor, then all of its automorphisms would be approximately inner \cite[Proposition 3.3]{ECfactors}.  In particular, by separable saturation, there would be a unitary $u\in N^\u$ that conjugates $j$ to the diagonal embedding.  Since $j$ factors through $\R^\u$, this contradicts \cite[Corollary 2.7]{ultraprodembed}.  Consequently, $N^\u$ is not an e.c. embeddable factor; since $N^\u\cong \R^\u$, neither is $\R^\u$.  
\end{proof}

\subsection{Building tracial von Neumann algebras by games}\label{games}  We now introduce a method for building tracial von Neumann algebras first introduced in \cite{enforceable} \linebreak (based on the discrete case presented in \cite{hodges}).  This method goes under many names, such as \textbf{Henkin constructions}, \textbf{model-theoretic forcing}, or \textbf{building models by games}.

We fix a countably infinite set $C$ of distinct symbols that are to represent generators of a separable tracial vNa that two players (traditionally named $\forall$ and $\exists$) are going to build together (albeit adversarially).
The two players take turns playing finite sets of expressions of the form $\left|\|p(c)\|_2-r\right|<\epsilon$, where $c$ is a tuple of variables from $C$, $p(c)$ is a $*$-polynomial, and each player's move is required to extend (that is, contain) the previous player's move.  These sets are called (open) \emph{conditions}.  The game begins with $\forall$'s move.  Moreover, these conditions are required to be \emph{satisfiable}, meaning that there should be some tracial von Neumann algebra $M$ and some tuple $a$ from $M$ such that $\left|\|p(a)\|_2-r\right|<\epsilon$ for each such expression in the condition.  We play this game for countably many rounds.
At the end of this game, we have enumerated some countable, satisfiable set of expressions. Provided that the players address a ``dense'' set of moments infinitely often, they can ensure that the play is \emph{definitive}, meaning that the final set of expressions yields complete information about all $*$-polynomials over the variables $C$ (that is, for each $*$-polynomial $p(c)$, there should be a unique $r$ such that the play of the game implies that $\|p(c)\|=r$) and that this data describes a countable, dense $*$-subalgebra of a unique tracial von Neumann algebra, which is called the \textbf{compiled structure}.

\begin{defn}
Given a property $P$ of tracial von Neumann algebras, we say that $P$ is an \textbf{enforceable} property is there a strategy for $\exists$ so that, regardless of player $\forall$'s moves, if $\exists$ follows the strategy, then the compiled structure will have property $P$.
\end{defn}

\begin{fact}\label{conjlemma}

\

\begin{enumerate}
    \item (Conjunction lemma \cite[Lemma 2.4]{enforceable})  If $P_n$ is an enforceable property for each $n\in \mathbb N$, then so is the conjunction $\bigwedge_n P_n$.
    \item (\cite[Proposition 2.10]{enforceable} Being e.c. is enforceable.  In particular, one can always enforce the compiled structure to be a locally universal McDuff II$_1$ factor.
\end{enumerate}
\end{fact}

\begin{defn}
A tracial von Neumann algebra $M$ is said to be \textbf{enforceable} if the property of being isomorphic to $M$ is an enforceable property.
\end{defn}

Clearly, if an enforceable tracial von Neumann algebra exists, then it is unique.

\begin{thm}(\cite[Theorem 5.2]{enforceable})
A positive solution to CEP is equivalent to $\R$ being the enforceable factor.
\end{thm}

Since a negative solution to CEP has recently been announced, it follows that $\R$ is not the enforceable factor.  That leaves the following open question:

\begin{question}
Is there an enforceable factor?
\end{question}

Since the enforceable factor, should it exist, is a ``canonical'' II$_1$ factor not isomorphic to $\R$, that leads these authors to guess that the above question has a negative answer.

Given any tracial von Neumann algebra $P$, there is a relative version of the above game where one restricts one's attention only to $P^\u$-embeddable algebras.  When playing this game, it is still the case that being an e.c. $P^\u$-embeddable algebra is enforceable.

Although we do not know if there is an enforceable factor, we do know that there is an enforceable embeddable factor:

\begin{thm}(\cite[Theorem 5.1]{enforceable})
$\R$ is the enforceable embeddable factor.
\end{thm}

\section{II$_1$ factors with the generalized Jung property}\label{jungprop}

\begin{conv*}
In the rest of this paper, unless stated otherwise, $M$, $N$, and $P$ denote separable II$_1$ factors.
\end{conv*}

\subsection{Definitions and first observations}\label{dfo}

\begin{defn}
We say that the pair $(N,M)$ is a \textbf{Jung pair} if $N$ embeds into $M^\u$ and any two embeddings of $N$ into $M^\u$ are unitarily equivalent.
We say that $M$ has the \textbf{Jung property} if $(M,M)$ is a Jung pair.
\end{defn}

As mentioned in the introduction, the starting point for this line of research is the following theorem of Jung (for which the property is named):

\begin{thm}\label{jung}(\cite{jung})
Assuming $N$ is embeddable, $(N, \R)$ is a Jung pair if and only if $N\cong \R$.
\end{thm}

In \cite{ultraprodembed}, the first- and third-named authors generalized Jung's theorem in several ways. One of them was the following observation:

\begin{thm}(\cite[Corollary 2.7]{ultraprodembed})\label{embedJung}\label{scottsrimain}
If $N$ is embeddable, then $N$ is a Jung factor if and only if $N\cong \R$.
\end{thm}

This result can be viewed as a special case of the more general results in \cite{ultraprodembed}. We document a short, but important separate proof for this result below, essentially following the proof of \cite[Corollary 2.7]{ultraprodembed}. The proof relies on the following:

\begin{lem}
Suppose that $M$, $N$, and $P$\footnote{In this lemma, we drop the assumption that $M$, $N$, and $P$ are II$_1$ factors and assume they are simply tracial von Neumann algebras, not even necessarily separable.} are such that there is an embedding $\sigma:N\hookrightarrow P$ and ucp maps $\phi:N\to M$ and $\psi:M\to P$ satisfying $\sigma=\psi\circ \phi$.  Furthermore assume that $M$ is injective.  Then $N$ is injective.
\end{lem}

\begin{proof}
It suffices to show that $\sigma(N)$ is injective.  Towards that end, fix operator systems $X\subseteq Y$ and a ucp map $f:X\to \sigma(N)$.  Let $\theta:\sigma(N)\to M$ be the ucp map $\theta(\sigma(x))=\phi(x)$.  Then $\theta\circ f:X\to M$ is a ucp map, whence, by the injectivity of $M$, there is a ucp extension $f':Y\to M$.  Letting $E:P\to \sigma(N)$ be the canonical conditional expectation, we see that $g:=E\circ \psi\circ f':B\to \sigma(N)$ is a ucp map extending $f$.
\end{proof}

Now suppose that $N$ is a separable II$_1$ factor for which there is an embedding $\sigma:N\hookrightarrow \R^\u$ with a ucp lift, that is, with a ucp map $\phi:N\to \ell^\infty(\R)$ such that $\sigma=Q\circ \phi$, where $Q:\ell^\infty(\R)\to \R^\u$ is the canonical quotient map.  Then since $\ell^\infty(\R)$ is injective, we are in the situation of the previous lemma, whence we can conclude that $N$ is injective.  By Connes' landmark theorem from \cite{connes}, $N\cong \R$.

Theorem \ref{embedJung} follows immediately from the previous paragraph.  Indeed, fix an embedding $\sigma:N\hookrightarrow \R^\u$ and view it as an embedding of $N\hookrightarrow N^\u$ (where, for notational simplicity, we are assuming that $\R\subseteq N$ is a concrete subfactor of $N$).  Since $N$ is a Jung factor, this embedding is unitarily equivalent to the diagonal embedding $N\hookrightarrow N^\u$, whence there are unitaries $u_k\in N$ such that $\sigma(x)=(u_kxu_k^*)_\u$ for all $x\in N$.  Setting $E:N\to \R$ to be the canonical conditional expectation and defining $\phi:N\to \ell^\infty(\R)$ by $\phi(x)=(E(u_kxu_k^*))_{k\in \mathbb N}$, we have that $\phi$ is a ucp lift of $\sigma$, whence $N\cong \R$.  Note that the same proof shows that if the above embedding $N\hookrightarrow N^\u$ is such that there is a sequence of ucp maps $\phi_k:N\to N$ for which $\sigma(x)=(E(\phi_k(x)))_\u$ for all $x\in N$, then $N\cong \R$.

The following result from \cite{ultraprodembed} is an even more serious generalization of Jung's theorem:

\begin{thm}\label{ucpconj}(\cite[Theorem 3.7]{ultraprodembed})
Suppose that $N$ is a separable embeddable II$_1$ factor for which, given any two embeddings $\pi_1,\pi_2:N\hookrightarrow \R^\u$, there is a sequence of ucp maps $\phi_k:\R\to \R$ for which $\pi_1(x)=(\phi_k(\pi_2(x)_k))_\u$.  Then $N\cong \R$.
\end{thm}

Using conditional expectations, one has the following:

\begin{cor}(\cite[Corollary 3.8]{ultraprodembed})
If $N$ is a separable embeddable II$_1$ factor, then for any II$_1$ factor $M$, one has that $(N,M)$ is a Jung pair if and only if $N\cong \R$. 
\end{cor}




In this paper, we will be concerned with an \emph{a priori} more general notion:

\begin{defn}
We say that the pair $(N,M)$ is a \textbf{generalized Jung pair} if $N$ embeds into $M^\u$ and any two embeddings of $N$ into $M^\u$ are automorphically equivalent.
We say that $M$ has the \textbf{generalized Jung property} if $(M,M)$ is a generalized Jung pair.
\end{defn}

It is clear that every Jung pair is a generalized Jung pair and every Jung factor is a generalized Jung factor.  We also note the following obvious fact:

\begin{lem}\label{fcpairsfcpair}
Suppose that $(N,M)$ is a generalized Jung pair.  Then $(N,M)$ is a factorial commutant pair if and only if it is a strong factorial commutant pair.
\end{lem}

Combining the previous lemma with Corollary \ref{Relcom}, we get:

\begin{cor}\label{popagenJungR}
For any separable embeddable II$_1$ factor $N$, we have that $(N,\R)$ is both a generalized Jung pair and a factorial commutant pair if and only if $N\cong \R$.
\end{cor}

The following lemma is also obvious but useful:

\begin{lem}\label{Jungee}
If $M_1\equiv M_2$, then for any $N$, $(N,M_1)$ is a generalized Jung pair if and only if $(N,M_2)$ is a generalized Jung pair.
\end{lem}

In the remainder of this section, we will be focused on II$_1$ factors with the generalized Jung property.  At the end of the paper, we will return to the notion of generalized Jung pairs.

The first hint that there is a connection between the generalized Jung property and model theory is the following:

\begin{lem}
If $M$ has the generalized Jung property, then $M$ is e.c. for $M^\u$-embeddable algebras.  In particular:
\begin{enumerate}
    \item If $M$ is locally universal and has the generalized Jung proeprty, then $M$ is e.c.
    \item If $M$ is embeddable and has the generalized Jung property, then $M$ is an e.c. embeddable factor.
\end{enumerate}
\end{lem}

\begin{proof}
Suppose that $M\subseteq N$ and $N$ embeds in $M^\u$.  Take an embedding $N\hookrightarrow M^\u$ and let $j:M\hookrightarrow M^\u$ be the composition.  Since $M$ has the generalized Jung property, there is an automorphism $\alpha$ of $M^\u$ that passes $j$ to the diagonal embedding.  It follows that $M$ is e.c. in $N$. 
\end{proof}

The following further indicates the link between the generalized Jung property and model theory:

\begin{thm}
Suppose that $N$ is a II$_1$ factor.  Then the following are equivalent:
\begin{enumerate}
\item $N$ has the generalized Jung property.
\item Every embedding of $N\hookrightarrow N^\u$ is elementary.
\item Whenever $j:N\hookrightarrow M$ is an embedding with $M\equiv N$, then $j$ is elementary.
\end{enumerate}
\end{thm}

\begin{proof}
(1) implies (2) follows immediately from the fact that the diagonal embedding is an elementary embedding.

(2) implies (3):  Assume that (2) holds and let $i:N\hookrightarrow M$ be an embedding, where $M\equiv N$.  Let $j:M\hookrightarrow N^\u$ be an elementary embedding, which exists by Fact \ref{ultrauniv}(2).  Then the composition $ji:N\hookrightarrow N^\u$ is elementary by assumption.  It follows that $i$ is also elementary by Fact \ref{elel}(2).

(3) implies (1). Assume that (3) holds and let $\pi_1,\pi_2:N\hookrightarrow N^\u$ be embeddings.  Then $\pi_1$ and $\pi_2$ are elementary embeddings by assumption.  In particular, the map $\pi_1(x)\mapsto \pi_2(x):\pi_1(N)\to \pi_2(N)$ is a partial elementary map between separable subalgebras of $N^\u$.  By Fact \ref{ultrahomog} above, there is an automorphism $\alpha$ of $N^\u$ extending this map. Thus $\alpha \circ \pi_1 = \pi_2$.
\end{proof}

\begin{remark}
The advantage of item (3) in the previous theorem is that it does not mention ultrapowers and does not appear to depend on set theory.\footnote{For the model theorists, one can rephrase (2) as $\operatorname{Th}(M)\cup \operatorname{Diag}(M)$ is complete.}. While the implications (1) implies (2) and (2) implies (3) do not depend on CH, the implication (3) implies (1) does use our standing CH assumption (via Fact \ref{ultrahomog}).  In fact, without assuming CH, there are a priori two definitions of generalized Jung property, one that holds for some nonprincipal ultrafilter on $\mathbb N$ and one that holds for all nonprincipal ultrafilters on $\mathbb N$.  It would be interesting to investigate if these two definitions coincide independent of the ambient set theory.
\end{remark}

We next discuss that the generalized Jung property is enforceable in the sense of \S\S\ref{games} above.  First, we need the following:

\begin{thm}(\cite[Theorem 2.14]{enforceable})
Fix a tracial von Neumann algebra $P$ (not necessarily a factor).  Given any sentence $\sigma$, there is a unique real number $r_{\sigma,P}$ such that, when playing the game restricted to $P^\u$-embeddable algebras, the property ``the compiled algebra $M$ satisfies $\sigma^M=r_{\sigma,P}$'' is an enforceable property.  
\end{thm}

When $P$ is locally universal, that is, when there is no restriction on the algebras, we write $r_\sigma$ instead of $r_{\sigma,P}$.

In what follows, we say that a tracial von Neumann algebra $N$ (not necessarily a factor) has the generalized Jung property if any two embeddings $N\hookrightarrow N^\u$ are automorphically equivalent.

The following definition is nonstandard but is useful for our purposes (see \cite[Proposition 3.10]{enforceable}).
\begin{defn}
Fix tracial von Neumann algebras $M$ and $P$ (not necessarily factors).  $M$ is called \textbf{finitely generic for $P^\u$-embeddable algebras} if:
\begin{enumerate}
\item $M$ has the generalized Jung property, and
\item $\sigma^M=r_{\sigma,P}$ for all sentences $\sigma$.
\end{enumerate}
When $P$ is locally universal, we simply call $M$ \textbf{finitely generic}.  When $P=\R$, we call $M$ \textbf{finitely generic embeddable}.
\end{defn}

Since $\R$ is the enforceable embeddable factor, it follows that any finitely generic embeddable algebra is elementarily equivalent to $\R$.

\begin{fact}\label{fgenf}(\cite[Proposition 3.9]{enforceable})
When playing the game restricted to $P^\u$-embeddable algebras, being finitely generic for $P^\u$-embeddable algebras is an enforceable property.  In particular, having the generalized Jung property is an enforceable property.
\end{fact}

We will also need the following fact:

\begin{fact}\label{fgec}(\cite[Corollary 3.11]{enforceable})
If $M$ is finitely generic for $P^\u$-embeddable algebras, then $M$ is e.c. for $P^\u$-embeddable algebras.
\end{fact}




%

\subsection{The case of embeddable factors}

In this subsection, we show that $\R$ is the unique embeddable factor with the generalized Jung property.

A first step towards this result is the following:

\begin{thm}\label{gJ implies ee R}
Suppose that $N$ is embeddable and has the generalized Jung property.  Then $N\equiv \R$.
\end{thm}

\begin{proof}
Fix embeddings $i:\R\hookrightarrow N$ and $j:N\hookrightarrow \R^\u$.  Consider the ultrapower maps $i^\u:\R^\u\hookrightarrow N^\u$ and $j^\u:N^\u\hookrightarrow (\R^\u)^\u$.  Notice that $j\circ i$ is elementary since $\R$ has the generalized Jung property and $i^\u \circ j$ is elementary since $N$ has the generalized Jung property.  By Fact \ref{ultrauniv}(1), $j^\u\circ i^\u=(j\circ i)^\u$ is also elementary as is $(i^\u)^\u \circ j^\u=(i^\u\circ j)^\u$. Consequently, we get a chain of iterated ultrapowers 
$$\R\hookrightarrow N\hookrightarrow \R^\u\hookrightarrow N^\u\hookrightarrow (\R^\u)^\u\hookrightarrow (N^\u)^\u\hookrightarrow\cdots$$
such that all maps between successive ultrapowers of $\R$ are elementary as are all maps between successive ultrapowers of $N$.  Setting $M$ to be the union of the chain, by Fact \ref{elel}(3) we see that $M$ is both an elementary extension of $N$ and $\R$, whence $N\equiv\R$.
\end{proof}

Now we prove the following general result, which is a modification (and simplification) of \cite[Proposition 4.12]{BCIexplained}.  The proof uses the material on types and heirs from \S\S\ref{types} above:

\begin{thm}\label{mainidea}
Suppose that $M$, $N$, and $P$ are such that:
\begin{enumerate}
\item $M\subseteq N\subseteq P^\u$,
\item $M\preceq P^\u$, and
\item $N'\cap P^\u$ is a factor.
\end{enumerate}
Then $M'\cap P^\u$ is a factor.
\end{thm}

\begin{proof}
Fix $a\in Z(M'\cap P^\u)$; we will show that $a\in \mathbb C$.  Let $p=\tp(a/M)$.  By Example \ref{typecommutant}, $p(P^\u)\subseteq Z(M'\cap P^\u)$.  Let $q\in S(N)$ be an heir of $p$, which exists by Fact \ref{heir}.  By Example \ref{heircommutant}, $q(P^\u)\subseteq N'\cap P^\u$.  Since $q(P^\u)\subseteq p(P^\u)\subseteq Z(M'\cap P^\u)$ and $N'\cap P^\u\subseteq M'\cap P^\u$, it follows that $q(P^\u)\subseteq Z(N'\cap P^\u)$.







Now take $b\in q(P^\u)$.  Since $N'\cap P^\u$ is a factor, we have that $b=\lambda\cdot 1$ for some $\lambda\in \mathbb C$.  Consequently, $d(x,\lambda\cdot 1)^p=d(x,\lambda\cdot 1)^q=d(b,\lambda\cdot 1)=0$, whence $a=\lambda\cdot 1$, as desired.
\end{proof}



\begin{cor}\label{eeRpopa}
If $M\equiv \R$, then any elementary embedding of $M$ into $\R^\u$ has factorial commutant, whence $(M,\R)$ is a factorial commutant pair.
\end{cor}

\begin{proof}
Fix an elementary embedding $j:M\hookrightarrow \R^\u$.  Then $j(M)\preceq \R^\u$.  By Theorem \ref{nate}, there is $N\subseteq \R^\u$ such that $j(M)\subseteq N$ and $N'\cap \R^\u$ is a factor.  By Theorem \ref{mainidea}, we have that $j(M)'\cap \R^\u$ is a factor, as desired.
\end{proof}

\begin{remark}
As mentioned in the introduction, a well-known open question of Popa asks whether or not every embeddable factor admits an embedding into $\R^\u$ with factorial commutant.  The previous corollary now gives continuum many non-isomorphic separable II$_1$ factors which satisfy the conclusion of Popa's question.
\end{remark}



We now arrive at one of the main results of this paper:

\begin{thm}\label{yay}
Suppose that $N$ is an embeddable generalized Jung factor.  Then $N\cong \R$.
\end{thm}

\begin{proof}
By Fact \ref{ultrauniv}(2), Theorem \ref{gJ implies ee R}, and Corollary \ref{eeRpopa}, we have that $(N,\R)$ is a factorial commutant pair.  By Theorem \ref{gJ implies ee R} and Lemma \ref{Jungee}, we have that $(N,\R)$ is a generalized Jung pair.  Thus, Corollary \ref{popagenJungR} implies $N\cong \R$.
\end{proof}

\begin{cor}
$\R$ is the unique finitely generic embeddable II$_1$ factor.
\end{cor}

\subsection{The general case}\label{generalcase}

In the recent preprint \cite{connessol}, a negative solution to the CEP was announced.  Working under the assumption that the proof there is correct, we immediately have:

\begin{thm}\label{nonembedGJ}
There is a non-embeddable factor with the generalized Jung property.
\end{thm}

\begin{proof}
If $M$ is a finitely generic factor (which exists by Fact \ref{fgenf}), then $M$ has the generalized Jung property and is non-embeddable (since it is locally universal by Facts \ref{ecfacts} and \ref{fgec}).
\end{proof}

A natural follow-up question is:  How many non-embeddable factors with the generalized Jung property are there?

\begin{conj}
There are continuum many non-isomorphic separable non-embeddable II$_1$ factors with the generalized Jung property.
\end{conj}

Some mild evidence for this conjecture is provided by the following:

\begin{thm}
If there are fewer than continuum many separable finitely generic factors, then the enforceable II$_1$ factor exists.
\end{thm}

\begin{proof}
This follows immediately from the so-called \textbf{Dichotomy theorem} \cite[Theorem 6.1]{enforceable} and the fact that being finitely generic is an enforceable property.
\end{proof}

As stated above, we believe that there does not exist an enforceable factor; consequently, that leads us to believe that there exist continuum many non-isomorphic separable finitely generic II$_1$ factors, and hence continuum many non-isomorphic separable II$_1$ factors with the generalized Jung property.

If we are willing to relax the demand that the generalized Jung algebra be a factor, then we can actually achieve continuum many non-isomorphic separable examples:

\begin{prop}
For every tracial von Neumann algebra $P$ (not necessarily a factor), there is a generalized Jung algebra $M$ such that $M$ and $P$ are mutually embeddable.
\end{prop}

\begin{proof}
The proof is identical to the proof of Theorem \ref{nonembedGJ} above relativized to the setting of $P^\u$-embeddable algebras.
\end{proof}

\begin{cor}
There are at least as many non-isomorphic generalized Jung algebras (resp. factors) as there are non-mutually embeddable algebras (resp. factors).
\end{cor}

It is an open problem whether or not there are continuum many non-mutually embeddable factors.  There are, however, continuum many non-mutually embeddable tracial von Neumann algebras:

\begin{thm}[Goldbring-Hart \cite{computability}]
There is a family $(M_t)_{t\in [0,1]}$ of tracial von Neumann algebras containing $\R$ such that the $M_t$'s are pairwise non-mutually embeddable.
\end{thm}

\begin{cor}
There are continuum many pairwise non-isomorphic non-embeddable generalized Jung algebras that contain $\R$.
\end{cor}

Thus far, all locally universal generalized Jung factors proven to exist have been finitely generic.  It is natural to wonder if this is always the case:

\begin{question}
Must a locally universal generalized Jung factor be finitely generic?
\end{question}

\begin{prop}
Suppose that $N\subseteq M$ are both generalized Jung and $N$ is locally universal.  Then $N\equiv M$.  In particular, if $N$ is finitely generic, then so is $M$.
\end{prop}

\begin{proof}
The proof of the first-statement is identical to the proof of Theorem \ref{gJ implies ee R}.  The second statement follows immediately from the first statement.
\end{proof}

\begin{cor}\label{enforceexists}
Suppose that the enforceable factor exists.  Then any locally universal generalized Jung factor is finitely generic.
\end{cor}

\begin{proof}
This follows from the previous proposition using the fact that the enforceable factor embeds into any e.c. factor (see \cite[Proposition 6.19 and Remark 6.29]{enforceable}).
\end{proof}

\begin{cor}
If there are fewer than continuum many generalized Jung factors, then every locally universal generalized Jung factor is finitely generic.
\end{cor}

Thus far, our arguments do not seem to apply to the following:

\begin{question}\label{nonembedJung}
Does there exist a non-embeddable Jung factor?
\end{question}

This question is of interest on its own, but a positive resolution of Question \ref{nonembedJung} would provide an example of a non-embeddable \tql self-tracially stable\tqr\ II$_1$ factor.

\begin{defn}
A II$_1$ factor $N$ is \textbf{self-tracially stable} if, for every embedding $\pi: N \hookrightarrow N^\u$, there is a sequence of embeddings $\pi_k: N\hookrightarrow N$ such that, for every $x \in N, \pi(x) = (\pi_k(x))_\u$.
\end{defn}

In \cite[Theorem 2.4]{ultraprodembed} it was shown that $\R$ is the only embeddable self-tracially stable II$_1$ factor.  
Thus it would be of significant interest to exhibit a non-embeddable self-tracially stable II$_1$ factor. The following proposition shows that if we were to exhibit a non-embeddable Jung factor, then we would automatically have an example of a non-embeddable self-tracially stable II$_1$ factor.

\begin{prop}
A Jung factor is self-tracially stable.
\end{prop}

\begin{proof}
Let $N$ be a Jung factor, and let an embedding $\pi: N \hookrightarrow N^u$ be given. Then $\pi$ is unitarily equivalent to the diagonal embedding, say by a unitary $u = (u_k)_\u \in N^\u$.  We can take $\pi_k: N\hookrightarrow N$ to be given by $\pi_k(x) = u_k^*xu_k$.
\end{proof}

We will have more to say about Question \ref{nonembedJung} in \S \ref{relres}.



\section{The space $\HOM_A(N,M^\u)$}\label{HOM}

The results of \cite{topdyn} show that if $N$ is a non-amenable embeddable factor, then $\HOM(N,\R^\u)$ is non-separable in the topology described in \S\S \ref{convex}.  Moreover, \cite[Theorem 3.23]{ultraprodembed} shows that if $N$ is a non-amenable embeddable factor, then for any ultraproduct of II$_1$ factors $\ds \prod_{k\rightarrow \u} M_k$, the space $\ds\HOM(N,\prod_{k\rightarrow \u}M_k)$ is non-separable.  



Proceeding by analogy, let $\HOM_A(N,M^\u)$ denote the space of all embeddings of $N$ into $M^\u$ modulo automorphic equivalence.  Given an embedding $\pi:N\hookrightarrow M^\u$, we let $[\pi]_A$ denote its class in $\HOM_A(N,M^\u)$.  
In this section, we show that $\HOM_A(N,M^\u)$ can be equipped with a complete metric analogous to the one for its counterpart $\HOM(N,M^\u)$.  


\subsection{The topometric structure on spaces of types}

First, for fixed separable II$_1$ factors $N$ and $M$ such that $N$ embeds into $M^\u$, we show that $\HOM_A(N,M^\u)$ is nothing more than an avatar of a familiar model-theoretic object, namely a space of types.  This will allow us to equip $\HOM_A(N,M^\u)$ with two topologies which interact in a very nice way.  

We first need to say a few words about types of countably infinite tuples.  Indeed, implicit in the definition of types in \S\S \ref{types} was that we were considering types of finite tuples.  However, the definition makes perfect sense for arbitrary tuples of variables as well.  Let $x=(x_1,x_2,\ldots)$ represent a countable sequence of variables and let $a=(a_1,a_2,\ldots,)$ denote a countable tuple from $M^\u$.  Given such a tuple $a$, we let $\tp(a):=\tp^{M^\u}(a/\emptyset)$ be just as in \S\S \ref{types}, now viewed as a function defined on all formulae $\varphi$ (without parameters) whose free variables are among those in $x$.  We let $S_\omega^M$ denote the set of such types.  As above, for $p\in S_\omega^M$, we say that $a\in M^\u$ \textbf{realizes} $p$ if $p=\tp(a)$.

We now describe two natural topologies on $S_\omega^M$.  First, to each formula $\varphi$ as above, let $K_\varphi$ denote a compact interval in $\mathbb R$ such that, for every tracial von Neumann algebra $N$, $\varphi^N$ takes values in $K_\varphi$.  Consequently, we may view $S_\omega^M$ as a subset of the compact spaces $\prod_\varphi K_\varphi$.  The topology on $S_\omega^M$ induced by this identification is referred to as the \textbf{logic topology} on $S_\omega^M$.  An alternative viewpoint on the logic topology is that it is the weakest topology making all maps $p\mapsto \varphi^p:S_\omega^M\to \mathbb R$ continuous (as $\varphi$ varies over all formulae). It is fairly straightforward to see that $S_\omega^M$ is a closed subset of $\prod_\varphi K_\varphi$, whence the logic topology on $S_\omega^M$ is compact.

There is another topology on $S_\omega^M$ induced from a  natural metric given by
$$d_{tp}(p,q):=\inf\{||a-b||_1 \ : \ a,b\in M^\u, p=\tp(a), q=\tp(b)\}$$ where 
$$||a-b||_1:=\sum_{j=1}^\infty 2^{-j}||a_j-b_j||_2$$ (if the reader will forgive the abuse in notation).  

The two topologies, while not necessarily the same, are connected in a way as formalized by the following definition:

\begin{defn}(Ben-Yaacov \cite{topometric})
A \textbf{topometric space} is a triple $(X,\mathcal T,d)$, where $(X,\mathcal T)$ is a topological space, $(X,d)$ is a metric space, and the following two axioms hold:
\begin{enumerate}
    \item The topology induced by the metric $d$ refines the topology $\mathcal T$.
    \item The metric $d:X\times X\to \mathbb R$ is $\mathcal T$-lower semicontinuous.
\end{enumerate}
We refer to $(X,\mathcal T)$ as the \textbf{topological reduct} of the topometric space and $(X,d)$ as the \textbf{metric reduct} of the topometric space.
\end{defn}

\begin{thm}(Ben Yaacov-Usvyatsov \cite[Fact 4.12]{BYU})
$S^M_\omega$, equipped with its logic topology and metric above, is a compact topometric space.\footnote{Technically the reference given is about type spaces of finite tuples, but the exact same proof works in the case of a countably infinite tuple.}
\end{thm}

It is important to note that when using ``topological'' adjectives when referring to a topometric space, these adjectives are being applied to the topological reduct.  Thus, in the above theorem, the compactness of the topometric space $S_\omega^M$ is referring to the compactness of the logic topology.

\subsection{Induced topometric structure and cardinality} We now turn to discuss how $\HOM_A(N,M^\u)$ inherits a topometric structure from $S_\omega^M$. As discussed in \S\S \ref{types}, we may alternatively view types as orbits in $M^\u$ under the action of $\operatorname{Aut}(M^\u)$.  We will need the following proposition:
\begin{prop}
Given a countably infinite tuple $a \in M^\u$, the orbit of $a$ under the action of $\text{Aut}(M^\u)$ is closed with respect to $||\cdot||_1$ on countable tuples in $M^\u$, and hence the metric $d_{tp}$ on $S_\omega^M$ is simply the distance between orbits.
\end{prop}

\begin{proof}
If $a_\lambda\to a$ and all $a_\lambda$ have the same type, then so does $a$.  Indeed, given any formula $\varphi(x)$, setting $r$ to be the common values of $\varphi^{M^\u}(a_\lambda)$, then since $\varphi^{M^\u}$ is a continuous function with respect to $||\cdot||_2$, we have that \[\varphi^{M^\u}(a)=\lim_\lambda \varphi^{M^\u}(a_\lambda)=r.\qedhere\]
\end{proof}

Now suppose that $a$ is a countable sequence from the unit ball of $N$ that generates $N$.  Suppose also that $\pi_1,\pi_2:N\hookrightarrow M^\u$ are given.  Then, by Fact \ref{sametype}\footnote{Again, the fact being referred to is about type spaces of finite tuples, but the proof works also for type spaces of countably infinite tuples.}, $[\pi_1]_A=[\pi_2]_A$ if and only if $\tp(\pi_1(a))=\tp(\pi_2(a))$.  It follows that we have a well-defined injection $\Phi_a:\HOM_A(N,M^\u)\hookrightarrow S_\omega^M$.  Moreover, the image of this injection is precisely those types which extend the \textbf{quantifier-free type} of $a$, denoted $\operatorname{qftp}(a)$, which is the restriction of $\tp(a)$ to the quantifier-free formulae.  It is easy to see that the set of such types is closed in the logic topology.  Thus, we have a bijection between $\HOM_A(N,M^\u)$ and a closed subset of $S_\omega^M$, allowing us to equip $\HOM_A(N,M^\u)$ with the structure of a compact topometric space.  Moreover, this topometric structure is canonical:

\begin{prop}
Suppose that $a$ and $b$ are both generators of $N$.  Then the map $$tp(\pi(a))\mapsto tp(\pi(b)):\Phi_a(\HOM_A(N,M^\u))\to \Phi_b(\HOM_A(N,M^\u))$$ is an isomorphism of topometric spaces, that is, the map is a homeomorphism between the topological reducts and a uniform homeomorphism between the metric reducts.
\end{prop}

\begin{proof}
We first show that the map is continuous in the logic topology.  To see this, suppose that $[\pi_\lambda(a)]_A\to [\pi(a)]_A$ in the logic topology and fix $\varphi(x)$ and $\epsilon>0$.  Take $\delta>0$ sufficiently small and let $p(x)$ be a $*$-polynomial such that $d(p(a),b)<\delta$.  Then $$|\varphi(\pi_\lambda(b))-\varphi(p(\pi_\lambda(a)))|,|\varphi(\pi(b))-\varphi(p(\pi(a)))|<\frac{\epsilon}{3}$$ if $\delta$ is sufficiently small.  For $\lambda\geq \lambda_0$, we have $|\varphi(p(\pi_\lambda(a)))-\varphi(p(\pi(a)))|< \frac{\epsilon}{3}$.  Consequently, for $\lambda\geq \lambda_0$, we have $|\varphi(\pi_\lambda(b))-\varphi(\pi(b))|<\epsilon$.  Since $\varphi$ and $\epsilon$ were arbitrary, we have that $[\pi_\lambda(b)]_A\to [\pi(b)]_A$, as desired.

We next show that the map is uniformly continuous.  Towards that end, fix $\epsilon>0$ and let $p(x)$ be a $*$-polynomial such that $d(p(a),b)<\frac{\epsilon}{3}$.  Suppose that $d(\pi_1(a),\pi_2(a))<\delta$.  For $\delta$ sufficiently small, it follows that $d(\pi_1(p(a)),\pi_2(p(a)))<\frac{\epsilon}{3}$, whence
$$d(\pi_1(b),\pi_2(b))\leq d(\pi_1(b),p(\pi_1(a)))+d(p(\pi_1(a)),p(\pi_2(a))+d(p(\pi_2(a)),\pi_2(b))<\epsilon.$$
It follows that the map is uniformly continuous.

By symmetry, we conclude that the map is both a homeomorphism between the topological reducts and a uniform homeomorphism between the metric reducts.
\end{proof}

The previous proposition allows us to equip $\HOM_A(N,M^\u)$ with an intrinsic structure of a compact topometric space.  We note that the metric on a compact topometric space is automatically complete (see \cite[Proposition 1.11]{topometric}), whence the metric on $\HOM_A(N,M^\u)$ inherited from $S_\omega^M$ is complete.\footnote{This is not hard to prove directly.}

In the unitary setting, $\HOM(N,M^\u)$ possesses the following metric\footnote{Brown in fact uses an $\ell^2$ version of the metric we present here; be assured that this choice is not important}: if $a = \left\{a_1,a_2,\dots\right\}\subset N$ is a sequence of contractions generating $N$, define

\begin{align*}
    d([\pi],[\rho]) &:= \inf_{u \in \u(M^\u)}\sum_{j=1}^\infty 2^{-j}||\pi(a_j) - u^*\rho(a_j)u||_2\\
    & = \inf_{u \in \u(M^\u)} ||\pi(a) - u^*\rho(a)u||_1.
\end{align*}

If one just naively adapts the metric from $\HOM(N,M^\u)$ to $\HOM_A(N,M^\u)$, one arrives at the following definition:

\begin{defn}
If $\left\{a_1,a_2,\dots\right\}\subset N$ is a sequence of contractions generating $N$, define

\[d_A([\pi],[\rho]) := \inf_{\alpha \in \text{Aut}(M^\u)}||\pi(a) - \alpha(\rho(a))||_1.\]
\end{defn}

\emph{A priori} the above metric induces the topology of \tql point-$||\cdot||_2$ convergence along representatives\tqr\ for weak approximate automorphic equivalence:

\begin{defn}
Two embeddings $\pi,\rho: N \rightarrow M$ are \textbf{weakly approximately automorphically equivalent} if there exists a sequence\footnote{Use a net if $N$ is not separable.} of automorphisms $\alpha_n \in \text{Aut}(M)$ such that, for every $\ds x \in N, \lim_{n\rightarrow \infty} ||\pi(x) - \alpha_n(\rho(x))||_2 = 0$.
\end{defn}

\begin{example}In general, weak approximate automorphic equivalence is strictly weaker than automorphic equivalence.  For example, let $\sigma: R\otimes R \rightarrow R$ be an isomorphism and consider the embeddings $\pi, \rho: \R\rightarrow \R$ given by $\pi(x) = x$ and $\rho(x) = \sigma(x\otimes 1)$ for every $x \in \R$.  It is well-known that any two embeddings of $\R$ into a II$_1$ factor are weakly approximately unitarily equivalent\footnote{This is weak approximate automorphic equivalence when all the automorphisms involved are inner.}. Thus $\pi$ and $\rho$ are weakly approximately automorphically equivalent, but $\pi$ and $\rho$ are not automorphically equivalent because the relative commutant of $\pi$ is $\mb{C}$ and the relative commutant of $\rho$ is $\sigma(\mb{C}\otimes \R)$.   
\end{example}

In contrast with the above example, we have the following result:

\begin{prop}
Let $N$ be a separable II$_1$ factor and let $M$ be a II$_1$ factor.  Then two embeddings $\pi,\rho: N \hookrightarrow M^\u$ are weakly approximately automorphically equivalent if and only if they are automorphically equivalent.
\end{prop}

\begin{proof}
There is only one nontrivial direction. Assume that $\pi$ and $\rho$ are weakly approximately automorphically equivalent.  Let $a$ denote an infinite tuple of generating contractions for $N$. Let $p=\tp(\rho(a))$.  Our hypothesis implies that $\pi(a)$ is in the closure of $p(M^\u)$, but since $p(M^\u)$ is closed, we have $\pi(a) \in p(M^\u)$, that is, $\pi$ and $\rho$ are automorphically equivalent.
\end{proof}

Note that this is the automorphic equivalence analog of Theorem 3.1 in \cite{autoultra}.

Clearly, for a fixed generator $a$ for $N$,  $d_{tp}\circ \Phi_a$ and $d_A$ agree on $\HOM_A(N,M^\u)$.

The following lemma is obvious:

\begin{lem}
The map $[\pi]\mapsto [\pi]_A:\HOM(N,M^\u)\to \HOM_A(N,M^\u)$ is a contraction.
\end{lem}

\begin{cor}
If $M$ is McDuff then $\HOM_A(N,M^\u)$ is pathconnected.  Consequently, when $(N,M)$ is not a Jung pair, then $|\HOM_A(N,M^\u)|\geq 2^{\aleph_0}$.
\end{cor}

In \cite[Theorem 4.7]{topdyn}, it was shown that if $\HOM(N,\R^\u)$ is either compact or separable (with respect to its usual metric topology), then $N\cong \R$.  While we do not know this yet in our context (though we suspect it to be true), we would like to point out two potentially relevant facts:

\begin{enumerate}
    \item A compact topometric spaces is such that its metric reduct is compact if and only if the two topologies coincide.
    \item If $X$ is a second countable, locally compact topometric space, then the metric reduct has density character $\leq \aleph_0$ or $2^{\aleph_0}$.  (See \cite[Proposition 3.20]{topometric}.)
\end{enumerate}

In regards to the second item, we note that $S_\omega^M$ is second-countable as one can restrict to a countable, ``dense'' set of formulae in the definition of the logic topology.

While it is unknown if $\HOM_A(N,M^\u)$ can ever be separable (besides being a point), we note that $S_\omega^M$ itself is never separable:

\begin{prop}
For any II$_1$ factor $M$, $S_\omega^M$ is non-separable.
\end{prop}

\begin{proof}
Suppose, towards a contradiction, that $S_\omega^M$ is separable.  By \cite[Proposition 1.16]{dfinite}, $\operatorname{Th}(M)$ would have a separable approximately $\omega$-saturated model $\tilde{M}$ and thus $\tilde{M}$ would embed all separable models of $\operatorname{Th}(M)$.  Recall now that there is a family $(N_\alpha)_{\alpha\in 2^\omega}$ of separable embeddable factors with the property that only countably many of them embed into any given separable factor.  (See \cite{NPS}.)  Since each $N_\alpha$ is embeddable, they also embed into $\tilde{M}^\u$, and thus into some separable model of $\operatorname{Th}(M)$, which in turn embeds into $\tilde{M}$.  In other words, each $N_\alpha$ embeds into $\tilde{M}$, which is a contradiction. 
\end{proof}

\section{Generalized Jung pairs of II$_1$ factors}\label{genpairs} 

In this section, we gather some collected observations and questions regarding generalized Jung pairs. 

\subsection{The ultimate generalization of Jung's theorem?}

We believe that the following is the main open question about generalized Jung pairs:

\begin{question}\label{bigquestion}
If $(N,\R)$ is a generalized Jung pair, is $N\cong \R$?
\end{question}

A positive answer to the previous question would be the ultimate generalization of Jung's original theorem.  By Lemma \ref{Jungee} and Theorem \ref{yay}, to give an affirmative answer to the above question, it would be enough to show that if $(N,\R)$ is a generalized Jung pair, then $N\equiv \R$.

One can view Theorem \ref{ucpconj} above as a partial solution to Question \ref{bigquestion}.  Indeed, that result says that if any two embeddings $N\hookrightarrow \R^\u$ are equivalent by an automorphism that has a ucp lift, then $N\cong \R$.

Note also that if Popa's question from the introduction has a positive solution, then by Corollary \ref{popagenJungR}, we have that $(N,\R)$ is a generalized Jung pair if and only if $N\cong \R$.  


Here is an even more basic question:

\begin{question}
If $(N,\R)$ is a generalized Jung pair, does $N$ have property \linebreak Gamma?
\end{question}

\subsection{Other collected musings on generalized Jung pairs}

We can, however, give plenty of examples of pairs that are not generalized Jung pairs.  Before doing so, we remind the reader of the following definition:

\begin{defn}\label{wspgap}
If $N$ is a subfactor of $M$, we say that $N$ has \textbf{w-spectral gap} in $M$ if $N'\cap M^\u=(N'\cap M)^\u.$
\end{defn}

For example, any property (T) II$_1$ factor has w-spectral gap in any II$_1$ factor extension.

We next point out the following recent theorem of the second-named author:

\begin{thm}\label{isaac2}(\cite[Corollary 2.9]{goldpopa})
If $M$ is an e.c. factor and $N$ is a w-spectral gap subfactor, then $N'\cap M^\u$ is a factor.
\end{thm}

We also utilize the following characterization of $\R$ (a strengthening of Fact \ref{Relcom} and \cite[Theorem 5.8]{saa} in the context of McDuff factors).

\begin{thm}\label{eMfe}
A separable embeddable II$_1$ factor $N$ is hyperfinite if and only if there exists a McDuff II$_1$ factor $M$ such that $(N,M)$ is a strong factorial commutant pair.

\end{thm}

\begin{proof}
If $N\cong \R$ then by \cite[Theorem 5.8]{saa}, $(N,M)$ is a strong factorial commutant pair.


On the other hand, if $(N,M)$ is a strong factorial commutant pair, we claim that there is only one embedding of $N$ into $M^\mc{U}$ up to unitary equivalence.  Indeed, by \cite{brocap} and \cite{saa}, $\HOM(N,M^\mc{U})$ is convex, and by Theorem \ref{exptchar}, every point is extreme.  It follows that $\HOM(N,M^\mc{U})$ is a singleton.  Then by \cite[Corollary 3.8]{ultraprodembed}, we have that $N\cong \R$. 
\end{proof}

\begin{cor}
Suppose that $M$ is an e.c. factor and $N$ is an embeddable w-spectral gap subfactor.  Then $(N,M)$ is not a generalized Jung pair.
\end{cor}

\begin{proof}
Suppose that $N$ is an embeddable w-spectral gap subfactor of the e.c. factor $M$ and yet, towards a contradiction, that $(N,M)$ is a generalized Jung pair.  By Theorem \ref{isaac2}, $(N,M)$ is a factorial commutant pair.  By Lemma \ref{fcpairsfcpair} and the contradiction assumption, we have that $(N,M)$ is a strong factorial commutant pair.  Since $M$ is McDuff (Fact \ref{ecfacts}(1)), we have that $N \cong\R$ by Theorem \ref{eMfe}.
\end{proof}

We end this section with the following deceptively difficult question David Sherman asked us: 

\begin{question}
Are there $N,M_1,M_2$ such that $N$ embeds into both $M_1^\u$ and $M_2^\u$, $(N,M_1)$ is a gen Jung pair, but $(N,M_2)$ is not?
\end{question}

\section{Super McDuff factors}\label{relres}

\subsection{First definitions and results}
Our work in \S \ref{jungprop} has some bearing on the notion of super McDuff factors, first introduced in \cite{DL} but not given a name until \cite{BCIexplained}:

\begin{defn}
A II$_1$ factor $M$ is \textbf{super McDuff} if $M'\cap M^{\u}$ is a II$_1$ factor.
\end{defn}

Recall that a factor $M$ is McDuff if $M'\cap M^\u$ is non abelian.  Super McDuffness requires moreover that the relative commutant is a factor.

\begin{examples}\label{superexamples}
The following II$_1$ factors are super McDuff:
\begin{enumerate}
    \item $\R$
    \item $L(\mathbb F_2\times S_\infty^{\operatorname{fin}})$
    \item $\bigotimes_\mb{N} L(\mathbb F_2)$
    \item $L(\breve{\mathbb F_2})$
\end{enumerate}
\end{examples}

These examples are \cite[Propositions 12, 19, 20]{DL} and \cite[Proposition 7]{ZM}.  Here, given a countable group $\Gamma$, $\breve{\Gamma}$ denotes a particular direct limit/semidirect product construction considered by Zeller-Meier in \cite{ZM}.  Also, $S_\infty^{\operatorname{fin}}$ denotes the group of permutations of $\mathbb N$ with finite support.

Not every McDuff factor is super McDuff:
\begin{example}\label{McSD}
  If $K$ is the group  constructed by Dixmier and Lance such that $L(K)$ has property Gamma but is not McDuff, then $L(K\times S_\infty^{\operatorname{fin}})$ is McDuff but not super McDuff \cite[Proposition 24]{DL}. 
\end{example}

After \cite{DL,ZM}, very little on super McDuff factors appeared in the literature until \cite[Corollary 4.10 and Proposition 4.12]{BCIexplained}, where the following two facts were proven:

\begin{fact}\label{braddisaac}

\

\begin{enumerate}
    \item If $\mathcal C$ is a separably saturated elementary extension of $M$, then $M$ is super McDuff if and only if $M'\cap \mathcal C$ is a factor.
    \item If $N\preceq M$ and $M$ is super McDuff, then $N$ is super McDuff.
\end{enumerate}
\end{fact}

We now use a recent observation of Ioana and Spaas to produce a large class of super McDuff II$_1$ factors.  Recall that $M$ is McDuff if and only if $M\cong M\otimes \R$.  Thus, if $P$ is non-McDuff, we can form a McDuff factor $P\otimes \R$ (non-isomorphic to $P$).  It is natural to wonder when such a factor is super McDuff.  We can completely answer that question:

\begin{prop}\label{superstrong}
If $P$ is not McDuff and $M=P\otimes \R$, then $M$ is super McDuff if and only if $P$ does not have property Gamma.
\end{prop}

\begin{proof}
This follows immediately from the recent observation \cite[Corollary 2.6]{ioaspa}, where they show that 
\[\mc{Z}(M'\cap M^\mc{U}) = \mc{Z}(P'\cap P^\mc{U}).\qedhere\]
\end{proof}

Notice that Examples \ref{superexamples}(2) and \ref{McSD} are special cases of the previous proposition.

Factors of the form $P\otimes \R$ with $P$ a non-Gamma factor are called \textbf{strongly McDuff} by Popa in \cite{spgap}.  Thus, the previous proposition shows that strongly McDuff factors are super McDuff.  In fact, this observation can be deduced from \cite[Theorem 4.7]{centralseqalg} which predates \cite{ioaspa}.







We next show how the proof that $\ds \bigotimes_\mb{N} L(\mb{F}_2)$ is super McDuff in \cite{DL} applies to all infinite tensor products of non-Gamma II$_1$ factors.\footnote{We note that a proof of this fact, which is Proposition \ref{itfullsuper} below, first appeared within the proof of Corollary G of \cite{marr}.} The following lemma provides a sufficient condition for being super McDuff.

\begin{lem}(\cite[Lemma 11]{DL})\label{DL}
Let $M$ be a finite factor, and let $E\subseteq M$ be a generating subset. For $j \in \mb{N}$, let $A_j\subseteq M$ be a nontrivial von Neumann subalgebra.  Suppose the following conditions hold:
	\begin{enumerate}
		\item For every finite subset $F\subseteq E$, there exists $j_0 \in \mb{N}$ such that $A_j$ commutes with $F$ for $j \geq j_0$;
		\item $\mc{Z}(M'\cap M^\u)\subseteq \bigcap_{j\in \mb{N}} A_j^\mc{U}$.
	\end{enumerate}
Then $M$ is super McDuff. 
\end{lem}

\begin{prop}\label{itfullsuper}
Suppose that, for each $n \in \mb{N}, M_n$ is a non-Gamma II$_1$ factor.  Setting $\ds M:= \bigotimes_\mb{N} M_n$, we have that $M$ is super McDuff.
\end{prop}

\begin{proof}
We wish to apply Lemma \ref{DL}.  Let $E\subseteq M$ be the set of all finite products of the respective sets of generators in their respective tensor position with 1's elsewhere.  For each $j \in \mb{N}$, let $A_j = \bigotimes_{n\geq j} M_n$ be the $j$th tail subalgebra of $M$. Clearly $E$ and $\left\{A_1,A_2,\dots\right\}$ satisfy condition (1) in Lemma \ref{DL}.

We next show that condition (2) in Lemma \ref{DL} is satisfied.  Fix $j \in \mb{N}$.  We will show that the stronger containment $M'\cap M^\u \subseteq \bigcap_{j\in\mb{N}} A_j^\u$ holds.  Since $M_n$ is non-Gamma for every $n \in \mb{N}$, we have that $M_1 \otimes \cdots \otimes M_{j-1}$ is non-Gamma by \cite{connes}.  Thus, by \cite[Proposition 3.2]{popagap}, $M_1 \otimes \cdots \otimes M_{j-1}$ has spectral gap in $M$.  Consequently, we have
\begin{align*}
M'\cap M^\u & \subseteq (M_1\otimes \cdots \otimes M_{j-1})'\cap M^\u\\
& = ((M_1\otimes \cdots \otimes M_{j-1})' \cap M)^\u\\
& = A_j^\mc{U}.
\end{align*} 
Thus, by Lemma \ref{DL}, $M$ is super McDuff.
\end{proof}

\begin{cor}\label{itstrongsuper}
Any infinite tensor product of strongly McDuff II$_1$ factors is super McDuff.
\end{cor}

Now we apply results of Marrakchi from \cite{marr} to deduce some more facts about super McDuff factors.

\begin{thm}\label{marrthmF}(\cite[Theorem F]{marr})
Let $M$ and $N$ be II$_1$ factors such that $(M\otimes N)'\cap (M\otimes N)^\u$ is a factor.\footnote{In other words, $M\otimes N$ is non-Gamma or super McDuff.}  Then $M'\cap M^\u$ and $N'\cap N^\u$ are also factors.
\end{thm}
Theorem \ref{marrthmF} implies that if $M\otimes N$ is super McDuff, then either both $M$ and $N$ are super McDuff, or (without loss of generality) $M$ is super McDuff and $N$ is non-Gamma.  In the latter case, one can rewrite the tensor decomposition by replacing $N$ with $N\otimes R$; in this decomposition, both tensor factors are indeed super McDuff. 

\begin{question}
Is the tensor product of two super McDuff factors again super McDuff?
\end{question}

Marrakchi's Theorem \ref{marrthmF} yields the following corollary generalizing the result from \cite{DL} that $\bigotimes_\mb{N} L(K)$ is not super McDuff:

\begin{cor}\label{marrcor}
Let $M$ and $N$ be II$_1$ factors such that $N$ has Gamma but is not super McDuff.  Then $M\otimes N$ is not super McDuff.
\end{cor}


\subsection{Connection to the current work}

In \cite{BCIexplained}, the following question was \linebreak raised:

\begin{question}\label{superconj}
If $M\equiv N$ and $N$ is super McDuff, is $M$ also super McDuff?
\end{question}

Given Fact \ref{braddisaac}(2), the previous question is the same as asking if $N\preceq M$ and $N$ is super McDuff, is $M$ also super McDuff?

Using our techniques from \S \ref{jungprop}, we can give a partial positive answer to Question \ref{superconj}.  First, we propose the following definition:

\begin{defn}
$M$ is said to have the \textbf{Brown property} if and only if:  whenever $N$ is a separable subfactor of $M^\u$, there is a separable subfactor $P$ of $M^\u$ such that $N\subseteq P$ and such that $P'\cap M^\u$ is a factor. 
\end{defn}

Theorem \ref{nate} can thus be restated as $\R$ has the Brown property (whence the nomenclature).

The following lemma is obvious but worth stating:

\begin{lem}
Suppose that $M$ has the Brown property and $N\equiv M$.  Then $N$ has the Brown property.
\end{lem}

The connection between the Brown property and being super McDuff is as follows:

\begin{prop}\label{Brownsuper}
$M$ has the Brown property if and only if:  for every $P\equiv M$, $P$ is super McDuff.
\end{prop}

\begin{proof}
First suppose that $M$ has the Brown property and $P\equiv M$.  Without loss of generality, we may assume that $P\subseteq M^\u$.  Since $M$ has the Brown property, there is $N\subseteq M^\u$ such that $P\subseteq N$ and $N'\cap M^\u$ is a factor.  By Theorem \ref{mainidea}, $P'\cap M^\u$ is a factor, whence, by Fact \ref{braddisaac}, $P$ is super McDuff.

Conversely, suppose that all $P$ elementarily equivalent to $M$ are super McDuff and take separable $N\subseteq M^\u$.  By the Downward Lowenheim-Skolem theorem (Fact \ref{DLS}), there is $P\preceq M^\u$ such that $N\subseteq P$.  By assumption and Fact \ref{braddisaac} again, $P'\cap M^\u$ is a factor.  Consequently, $M$ has the Brown property.
\end{proof}

\begin{remark}
If Question \ref{superconj} has a positive answer, then the previous proposition shows that the Brown property is the same as being super McDuff.
\end{remark}

Proposition \ref{Brownsuper} and Theorem \ref{nate} immediately imply:

\begin{cor}\label{Rsuperee}
If $M\equiv \R$, then $M$ is super McDuff.
\end{cor}


We next show that Corollary \ref{Rsuperee} does not follow from Proposition \ref{superstrong}.  First, we need a pair of definitions. The following definition first appeared in \cite{sakai0}.

\begin{defn}
A II$_1$ factor $M$ is called \textbf{asymptotically commutative (AC)} if there exists a sequence of automorphisms $\left\{\alpha_n\right\}$ from $\text{Aut}(M)$ such that \linebreak $\ds \lim_{n\rightarrow \infty}||[\alpha_n(a),b]||_2=0$ for every $a,b \in M$.  This is equivalent to the existence of an automorphism $\alpha \in \text{Aut}(M^\u)$ of the form $\alpha((x_k)_\u) = (\alpha_k(x_k))_\u$ for $\alpha_k \in \text{Aut}(M)$ such that $\alpha(M) \subset M'\cap M^\u$.
\end{defn}

\begin{defn}
$M$ is \textbf{inner asymptotically central (IAC)} if, for any finite sets $F,G\subseteq M$ and any $\epsilon>0$, there is a unitary $u\in M$ such that $\|[uxu^*,y]\|_2<\epsilon$ for all $x\in F$ and $y\in G$. In other words, $M$ is IAC if and only if there is a unitary $u\in M^\u$ such that $uMu^*\subseteq M'\cap M^\u$.
\end{defn}

Clearly, AC factors are McDuff. Note that the property IAC is an instance of AC where all the witnessing automorphisms are inner.  By \cite[Proposition 1]{sakai0}, any infinite tensor product of a finite factor is AC.

\begin{example}\label{RIAC}
  $\R$ is IAC.  (See \cite[Proposition 3]{ZM}.)
\end{example}

The following is immediate from \cite[Proposition 4.8]{BCIexplained}:

\begin{fact}\label{IAClocal}
If $M\equiv N$ and $M$ is IAC, then so is $N$.
\end{fact}

Consequently, if $M\equiv \R$, then $M$ is IAC.  On the other hand, we have the following recent observation of Adrian Ioana.  We thank him for giving us permission to include a proof of his result here.

\begin{prop}\label{adrian} Strongly McDuff factors are not AC.


\end{prop}

\begin{proof}
Suppose that $N$ is a II$_1$ factor without property Gamma and $M=N\otimes \R$.  Suppose, towards a contradiction, that $M$ is AC.  Then there is an automorphism $\alpha \in \text{Aut}(M^\u)$ such that $\alpha(M)\subseteq M'\cap M^\u$ and $\alpha(x) = (\alpha_k(x))_\u$ where $\alpha_k \in \text{Aut}(M)$.  Since $N$ has spectral gap in $M$, we have that $M'\cap M^\u\subseteq N'\cap M^\u=(N'\cap M)^\u=\R^\u$.  Consequently, $\alpha(M)\subseteq \R^\u$.  Letting $E_\R:M\to \R$ denote the canonical conditional expectation, it follows, after passing to a subsequence of $(\alpha_k)$, that $\|\alpha_k(x)-E_\R(\alpha_k(x))\|_2\to 0$ for all $x\in M$.  Set $\R_k:=\alpha_k^{-1}(R)\subseteq M$ and let $E_{\R_k}:M\to \R_k$ denote the conditional expectation.  It follows that $\|x-E_{\R_k}(x)\|_2\to 0$ for all $x\in M$.  Since each $\R_k$ is amenable, it follows from \cite[Corollary 2.5]{exotic} that $M$ is amenable, which is a contradiction.
\end{proof}

The following is an immediate consequence of Ioana's result and Fact \ref{IAClocal}:

\begin{cor}
If $M\equiv \R$, then $M$ is not strongly McDuff.
\end{cor}



We now present two further corollaries of Corollary \ref{Rsuperee}.  First, Corollary \ref{Rsuperee} and the Downward Lowenheim-Skolem immediately implies:

\begin{cor}\label{superextension}
Any embeddable factor is contained in a super McDuff factor.
\end{cor}

The original motivation for Question \ref{superconj} was to obtain the following corollary:  

\begin{cor}
$L((\mathbb F_2\times \mathbb Z)^{\breve{}})\not\equiv \R$.
\end{cor}

In \cite{ZM}, it is shown that $L((\mathbb F_2\times \mathbb Z)^{\breve{}})$ is not super McDuff, whence the result follows from Corollary \ref{Rsuperee} above.  Since $L((\mathbb F_2\times \mathbb Z)^{\breve{}})$ is McDuff and IAC, before Corollary \ref{Rsuperee}, we did not know a method of distinguishing this factor from $\R$ in a first-order fashion.  Before the appearance of \cite{noniso}, not many non-elementarily equivalent II$_1$ factors were known, hence the interest in establishing the previous corollary.

\subsection{Are e.c. factors super McDuff?}

Recall that all e.c. (embeddable) factors are McDuff.  This raises:

\begin{question}\label{ecsuperquestion}
Are all e.c. (embeddable) factors super McDuff?
\end{question}

If the answer to the above question is ``no'' for embeddable factors, then there is an e.c. embeddable factor $M$ such that $M\not\equiv \R$.  This would be very interesting as, at present, it is not known if all e.c. (embeddable) factors are elementarily equivalent or not.  Another interesting consequence of a negative answer to the previous question would be that it is not true that being super McDuff is closed under existential substructure (as it is for elementary substructure), for there would be an e.c. embeddable non-super McDuff factor which is contained in (and thus e.c. in) a super McDuff factor by Corollary \ref{superextension} above.  

One might wonder if one could answer Question \ref{ecsuperquestion} by showing that e.c. (embeddable) factors are strongly McDuff and then quote Proposition \ref{superstrong}.  This is unfortunately not the case:

\begin{prop}\label{ecIAC}
E.c. (embeddable) factors are IAC.  Consequently, e.c. (embeddable) factors are never strongly McDuff.
\end{prop}

\begin{proof}
Suppose that $M$ is an e.c. (embeddable) factor.  Let $\alpha$ be the flip automorphism of $M\otimes M$, that is, $\alpha(x\otimes y)=y\otimes x$ for all $x,y\in M$, and let $N:=(M\otimes M)\rtimes_\alpha \mathbb Z$.  View $M$ as contained in $N$ via the map $x\mapsto x\otimes 1$.  Then fixing an embedding $N\hookrightarrow M^\u$ that restricts to the diagonal embedding $M\hookrightarrow M^\u$, we have a unitary $u\in M^\u$ such that $uMu^*\subseteq M'\cap M^\u$, whence $M$ is IAC.
\end{proof}





While there is an example of an e.c. embeddable super McDuff factor, namely $\R$, the following weakening of Question \ref{ecsuperquestion} in the non-embeddable case eludes us:

\begin{question}\label{existsecsuper}
Is there an e.c. super McDuff factor?
\end{question}

One may even ask whether or not for certain subclasses of the class of e.c. factors one has a positive solution to Question \ref{existsecsuper}.  For example:

\begin{question}\label{existsgenlocsuper}
Is there a generalized Jung, locally universal factor that is super McDuff?
\end{question}

Note that since a generalized Jung, locally universal factor $M$ is McDuff, we have that such an $M$ is super McDuff if and only if $(M,M)$ is a factorial commutant pair if and only if $(M,M)$ is a strong factorial commutant pair.

The previous question has bearing on Question \ref{nonembedJung}:

\begin{prop}
Suppose that there is a generalized Jung, locally universal, super McDuff factor $M$.  Then $M$ is a non-embeddable Jung factor.
\end{prop}

\begin{proof}
Since $M$ is generalized Jung and super McDuff, $(M,M)$ is a strong factorial commutant pair.  By Theorem \ref{exptchar} above, every point of $\HOM(M,M^\u)$ is extreme, whence it is a singleton.  In other words, $M$ is a Jung factor.  Since $M$ is locally universal, $M$ is non-embeddable. 
\end{proof}

\begin{cor}
If all e.c. factors are super McDuff, then for locally universal factors, the notions ``generalized Jung factor'' and ''Jung factor'' coincide.
\end{cor}

Specializing Question \ref{existsgenlocsuper} even further:

\begin{question}\label{fingensuper}
Is there a finitely generic super McDuff factor?
\end{question}

Recall that we do not know if every generalized Jung locally universal factor is finitely generic but that these notions do coincide if the enforceable factor exists (Corollary \ref{enforceexists}).

By the Conjunction Lemma, a positive answer to the next question would imply a positive answer to the previous question.

\begin{question}
Is being super McDuff an enforceable property?
\end{question}

\begin{prop}
Suppose that the enforceable factor exists.  Then being super McDuff is enforceable if and only if some finitely generic factor is super McDuff. 
\end{prop}

\begin{proof}
The forward direction is immediate.  For the converse, let $\mathcal E$ be the enforceable factor and suppose that $M$ is a finitely generic super McDuff factor.  Then $\mathcal E$ embeds in $M$ (again by \cite[Proposition 6.19 and Remark 6.29]{enforceable}) and thus $\mathcal E\preceq M$ since $\mathcal E$ is finitely generic.  By Fact \ref{braddisaac}(2), it follows that $\mathcal E$ is super McDuff, whence being super McDuff is an enforceable property.
\end{proof}

\begin{question}\label{superchain}
Is the union of a chain of super McDuff factors once again super McDuff?
\end{question}

\begin{remark}
If being super McDuff is actually axiomatizable, that is, closed under ultraproducts, but Question \ref{ecsuperquestion} has a negative answer, then Question \ref{superchain} has a negative answer.\footnote{An axiomatizable property closed under unions of chains is axiomatized by $\forall\exists$-sentences, that is, sentences of the form $\sup_x\inf_y \varphi(x,y)$ where $\varphi$ has no quantifiers; see \cite[Proposition 2.4.4.(3)]{munster}.  If this property further satisfies that every factor embeds into a factor with this property, as we established for super McDuffness in Corollary \ref{superextension}, then every e.c. factor has this property; see \cite{ECfactors}.}
\end{remark}

We can give a partial answer to Question \ref{superchain}, but we first need the following technical result:

\begin{prop}\label{technical}
Let $\left\{M_n\right\}$ be an increasing sequence of separable subfactors of $M^\u$, each with factorial relative commutant.  Let $P$ denote the union of the $M_n$'s.  Then $P'\cap M^\u$ is a factor.
\end{prop}

\begin{proof}
By Fact \ref{sametrace}, it suffices to show that if $p$ and $q$ are projections in $P'\cap M^\u$ of the same trace, then there is a unitary $u \in P' \cap M^\u$ such that $p = u^*qu.$ Let $p,q \in P'\cap M^\u$ be two projections of the same trace.  For each $n \in \mathbb N$ we have $p,q \in M_n'\cap M^\u$.  So for each $n \in \mathbb N$, there exists a unitary $u_n \in M_n'\cap M^\u$ such that $qu_n = u_nq$.  Let $\left\{x_j\right\}$ be a countable generating subset of $P$ such that $x_j \in M_{n_j}$ for some $n_j \in \mathbb{N}$.  Let $p = (p^{(k)})_\u, q = (q^{(k)})_\u, u_n = (u_n^{(k)})_\u,$ and $x_j = (x_j^{(k)})_\u$.  By Fact \ref{stable}, we can and do choose $p^{(k)}$ and $q^{(k)}$ to be projections all of the same trace and $u^{(k)}_{m_n}$ to be unitaries for every $k,n \in \mb{N}$.  Form a(n increasing) sequence $\left\{m_n\right\}$ such that for each $n\in \mathbb{N}, \left\{x_1,\dots, x_n\right\}\subset M_{m_n}$.  Thus 
\[||[x_j,u_{m_n}]||_2  = \lim_{k\rightarrow \u} ||[x_j^{(k)},u_{m_n}^{(k)}]||_2= 0\]for every $1 \leq j \leq n$.  Also, 
\[||qu_{m_n} - u_{m_n}p||_2 = \lim_{k\rightarrow \u} ||q^{(k)}u_{m_n}^{(k)} - u_{m_n}^{(k)}p^{(k)}||_2 = 0.\]  So, for $n\in \mathbb{N}$, 
\[\lim_{k\rightarrow \u} \sum_{j \in \mathbb{N}} 2^{-j} ||[x_j^{(k)},u_{m_n}^{(k)}]||_2 \leq 2^{-n+1}.\]

Now we produce a single unitary $u = (u^{(k)})_\u$ in $P'\cap M^\u$ that intertwines $p$ and $q$.  For $k \in \mathbb{N}$, let $u^{(k)} \in \left\{u_{m_1}^{(k)},\dots,u_{m_k}^{(k)}\right\}$ be such that 
\begin{align*}
\left(\sum_{j\in \mathbb{N}} 2^{-j} ||[x_j^{(k)},u^{(k)}]||_2\right) + ||q^{(k)}u^{(k)} - u^{(k)}p^{(k)}||_2 \\
= \min_{1\leq n\leq k} \left[\left(\sum_{j\in \mathbb{N}} 2^{-j}||[x_j^{(k)},u_{m_n}^{(k)}]||_2\right) + ||q^{(k)}u_{m_n}^{(k)} - u_{m_n}^{(k)}p^{(k)}||_2\right].
\end{align*}

Then, for every $n \in \mathbb{N}$, we have 
\begin{align*}
&\lim_{k\rightarrow \u} \left(\sum_{j\in\mathbb{N}} 2^{-j}||[x_j^{(k)},u^{(k)}]||_2\right) + ||q^{(k)}u^{(k)} - u^{(k)}p^{(k)}||_2\\
&\leq \lim_{k\rightarrow \u} \left(\sum_{j \in \mathbb{N}} 2^{-j}||[x_j^{(k)},u_{m_n}^{(k)}]||_2\right) + ||q^{(k)}u_{m_n}^{(k)} - u_{m_n}^{(k)}p^{(k)}||_2\\
&\leq 2^{-n+1}.
\end{align*}

Hence,
\begin{align*}
&\left(\sum_{j\in\mathbb{N}} 2^{-j}||[x_j,u]||_2\right) + ||qu - up||_2 \\
&=\lim_{k\rightarrow \u} \left(\sum_{j\in\mathbb{N}} 2^{-j}||[x_j^{(k)},u^{(k)}]||_2\right) + ||q^{(k)}u^{(k)} - u^{(k)}p^{(k)}||_2\\
&=0,
\end{align*}
and it follows that $u \in P'\cap M^\u$ and $u^*qu = p$.
\end{proof}

\begin{cor}
Suppose that $M_1\preceq M_2\preceq \cdots$ is an elementary chain of super McDuff II$_1$ factors.  Let $P$ denote the union of the $M_n$'s.  Then $P$ is also super McDuff.
\end{cor}

\begin{proof}
Since each $M_n$ is super McDuff, using Fact \ref{braddisaac}, we see that the relative commutant $M_n'\cap P^\u$ is a factor. Proposition \ref{technical}, it follows that $P'\cap P^\u$ is a factor, that is, $P$ is super McDuff.
\end{proof}

As further partial evidence that the answer to Question \ref{superchain} could be \tql yes,\tqr\ recall that Corollary \ref{itstrongsuper} says that any infinite tensor product of strongly McDuff factors (a subclass of super McDuff factors) is super McDuff.  On the other hand, the fact that $\bigotimes_\mb{N} L(\mb{F}_2)$ is super McDuff (see \cite{DL}) and more generally Proposition \ref{itfullsuper} show non-super McDuffness is not preserved by increasing chains.




For the last results of this section, we need to introduce the class of infinitely generic factors:

\begin{fact}(\cite[Propositions 5.7, 5.10, and 5.14]{ECfactors})
There is a class $\mathcal G$ of II$_1$ factors satisfying the following three properties:
\begin{enumerate}
    \item Every II$_1$ factor is contained in an element of $\mathcal G$.
    \item If $M_1,M_2\in \mathcal G$ and $M_1\subseteq M_2$, then $M_1\preceq M_2$.
    \item $\mathcal G$ is the maximum class with properties (1) and (2).
\end{enumerate}
\end{fact}

Elements of $\mathcal G$ are called \textbf{infinitely generic} II$_1$ factors.

\begin{fact}(\cite[Proposition 5.11, Proposition 5.17, and Lemma 5.20]{ECfactors})\label{infgenfacts}
\begin{enumerate}
    \item Infinitely generic factors are e.c.
       \item If $M\preceq N$ and $N$ is infinitely generic, then so is $N$.
    \item If $M_1$ and $M_2$ are infinitely generic, then $M_1\equiv M_2$.

\end{enumerate}
\end{fact}

As with finitely generic factors, it makes sense to relativize and consider the class of infinitely generic embeddable factors.  

The following fact is well-known in the discrete case.  Since it has yet to be observed in the continuous case, we include a proof here:

\begin{lem}\label{infgenchain}
The union of a chain of infinitely generic (embeddable) factors is infinitely generic (embeddable).
\end{lem}

\begin{proof}
Let
$$M_0\subseteq M_1\subseteq M_2\subseteq \cdots$$ be a chain of infinitely generic (embeddable) factors with union $M$.  Note that this chain is elementary, whence $M_k\preceq M$ for each $k$. Let $N$ be an infinitely generic factor containing $M$.  By Fact \ref{infgenfacts}(2), it suffices to show that $M\preceq N$.  To see this, it suffices to show, for any formula $\varphi(x)$, any $k$, and any $a\in M_k$, that $\varphi(a)^{M}=\varphi(a)^N$.  Since $M_k\preceq M$, we have that $\varphi(a)^M=\varphi(a)^{M_k}$.  Since $M_k\preceq N$, we have that $\varphi(a)^N=\varphi(a)^{M_k}$.
\end{proof}

A positive answer to Question \ref{superchain} shows that Question \ref{ecsuperquestion} has a positive answer for a large collection of e.c. (embeddable) factors:

\begin{prop}
Suppose that being super McDuff is closed under unions of chains.  Then every infinitely generic (embeddable) factor is super McDuff.
\end{prop}

\begin{proof}
Let $M$ be infinitely generic (embeddable).  Consider the chain
$$M=M_0\subseteq N_0\subseteq M_1\subseteq N_1\subseteq \cdots$$ with each $M_i$ infinitely generic (embeddable) and each $N_i$ super McDuff (embeddable).  Note that such a chain exists since every (embeddable) factor embeds into an (embeddable) super McDuff factor.\footnote{Indeed, given (embeddable) $P$, first embed $P$ into an (embeddable) non-Gamma factor, say $P*L(\mathbb F_2)$, then tensor with $\R$.}  Let $Q$ be the union of this chain.  Then $Q$ is super McDuff by assumption and infinitely generic (embeddable) by Lemma \ref{infgenchain}.  However, since $M$ is also infinitely generic (embeddable), we have that $M\preceq Q$, and thus $M$ is super McDuff by Fact \ref{braddisaac}(2). 
\end{proof}

Recently, the second-named author proved the following:
\begin{thm}\label{isaac}(\cite[Theorem 2.18]{goldpopa})
Let $M$ be an infinitely generic factor.  Then for any property (T) factor $N$, $(N,M)$ is a factorial commutant pair.
\end{thm}

\begin{prop}
Suppose that P satisfies the following three properties:
\begin{enumerate}
\item $P$ is elementarily equivalent to the infinitely generic factors;
\item $P$ has the generalized Jung property;
\item $P$ is contained in a property (T) factor.
\end{enumerate}
Then P is super McDuff.
\end{prop}

\begin{proof}
Fix an infinitely generic factor $M$; by (1), $P\equiv M$.  By (3), we may take a property (T) factor $N$ such that $P\subseteq N$.  By Theorem \ref{isaac}, there is an embedding $N\hookrightarrow M^\u$ with factorial relative commutant.  The restriction $P\hookrightarrow M^\u$ is elementary by (1) and (2).  Thus, by Theorem \ref{mainidea} above, $P'\cap M^\u$ is a factor.  By Fact \ref{braddisaac}(1), we have that $P$ is super McDuff.
\end{proof}

The following question is open:

\begin{question}\label{fininf}
If $M$ is a finitely generic factor and $N$ is an infinitely generic factor, is $M\equiv N$?
\end{question}

If the answer to the previous question is ``no'', then once again we have non-elementarily equivalent e.c. factors.  Otherwise, if $P$ is a finitely generic factor, then $P$ satisfies (1) and (2) in the previous theorem.  Concerning item (3), the following seems to be open:

\begin{question}\label{containedinT}
Is every separable II$_1$ factor contained in a property (T) factor?
\end{question}

Returning to Question \ref{fingensuper}:

\begin{cor}
Suppose the answer to Questions \ref{fininf} and \ref{containedinT} are both positive.  Then any finitely generic II$_1$ factor is super McDuff.
\end{cor}







\subsection{Landscape}

We include the diagram below to provide a snapshot of the current landscape of the numerous properties a McDuff II$_1$ factor can have.



 The abbreviation \tql e.e.\tqr stands for \tql elementary equivalent.\tqr  A region is shaded gray if it is unknown whether or not there exists an algebra with the corresponding properties.  A dashed line connecting an algebra to a region indicates that it is unknown if that algebra has the properties corresponding to that region.

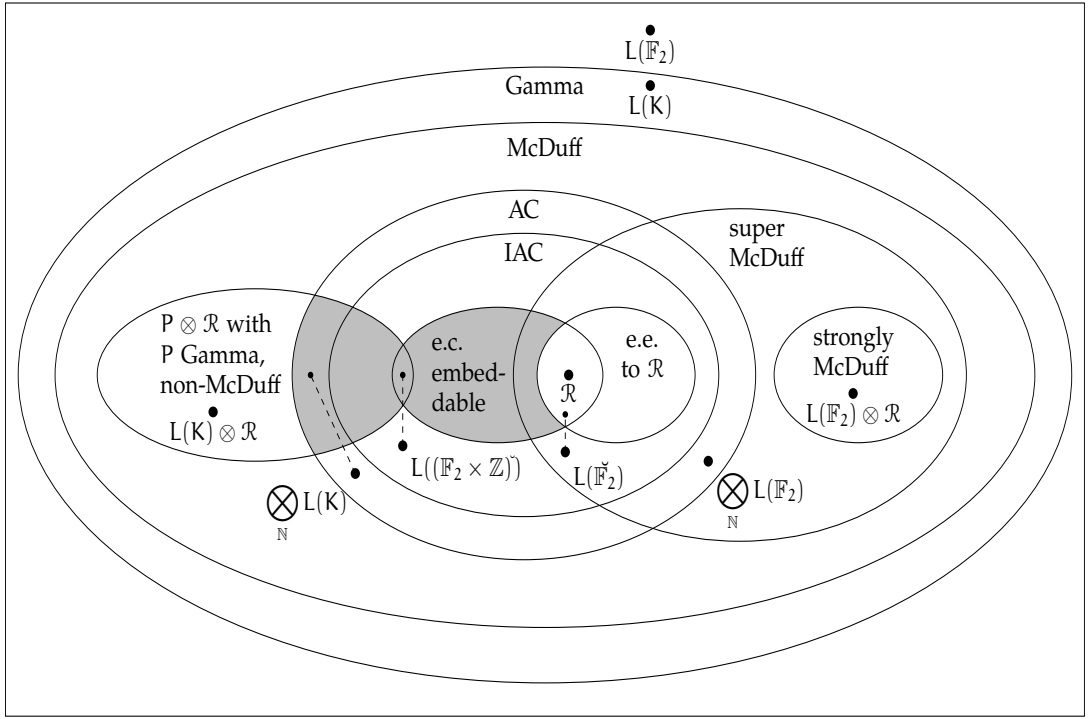
\begin{center}

\resizebox{14.5cm}{9.5cm}{

\begin{tikzpicture}

	
	\begin{scope}
		\clip \WeaklyMcDuff;
		\fill[lightgray] \AC;
	\end{scope}
	
	
	\begin{scope}
		\clip (-9,-3) rectangle (9,3)
			\EER;
		\fill[lightgray] \EC;
	\end{scope}

	\draw \McDuff;
	\draw \PGamma;
	\draw \EER;
	\draw \SuperMcDuff;
	\draw \StronglyMcDuff;
	\draw \AC;

	\draw \WeaklyMcDuff;

	\draw \IAC;
	\draw \EC;

	\fill (-4.45,0) circle[radius=1.5pt];
	\fill (-2.7,0) circle[radius=1.5pt];
	
	\draw[dashed] (-3.6,-1.6) -- (-4.45,0);
	
	\draw[dashed] (-2.7,-1.15) -- (-2.7,0);

	\draw[dashed] (0.3858,-1.25) -- (0.3858,-0.64);
	\fill (0.3858,-0.64) circle[radius=1.5pt];
	

	\node at (0,3.7) {McDuff};
	\node at (0, 4.7) {Gamma};
	\node at (1.85,0.3) [text width = 1.55cm, align=center]  {e.e. to  $\mc{R}$};
	\node at (5.85,.4) [text width = 1.5cm] {strongly McDuff};
	\node at (4.25,2.1) [text width = 1.5cm] {super McDuff};
	\node at (-0.4, 2.7) {AC};
	\node at (-.4, 2.0) {IAC};	
	\node at (-1.4,0) [text width = 1.5cm] {e.c. embeddable};

	\node at (-6.05,0.35) [text width = 2.5cm]  {$P\otimes \R$ with $P$ Gamma, non-McDuff};
	
	\fill (0.45,0) circle[radius=2.5pt]; 
	\node at (0.45,0) [below]{$\R$};
	\fill (2,5.6) circle[radius=2.5pt];
	\node at (2,5.6)[below] {$ L(\mb{F}_2)$};
	\fill (2,4.7) circle[radius=2.5pt];
	\node at (2, 4.7) [below]{$L(K)$ };
	\fill (-6.3,-0.6) circle[radius=2.5pt];
	\node at (-6.3, -0.6) [below]{$ L(K) \otimes \R$};
	\fill (-2.7,-1.15) circle[radius=2.5pt];
	\node at (-2.7, -1.15) [below right]{$ L((\mathbb F_2\times \mathbb Z)^{\breve{}})$ };
	\fill (5.85,-.3) circle[radius=2.5pt];
	\node at (5.85,-.3)[below] {$L(\mb{F}_2)\otimes \R$};
	\fill (3.1,-1.4) circle[radius=2.5pt];
	\node at (3.1,-1.4)[below right] {$\ds \bigotimes_\mb{N} L(\mb{F}_2)$ };
	\fill (0.3858,-1.25) circle[radius=2.5pt];
	\node at (0.3858,-1.25) [below right]{$ L(\breve{\mb{F}_2})$};

	\fill (-3.6,-1.6) circle[radius=2.5pt];
	\node at (-3.6,-1.6) [below left] {$\ds\bigotimes_\mb{N} L(K)$};

	\node at (0,-5.4) {};
	\node at (10.1,0) {};
	\node at (-10.1,0) {};
	\node at (0, 5.9) {};
	\draw (current bounding box.north east) -- (current bounding box.north west) -- (current bounding box.south west) -- (current bounding box.south east) -- cycle;
\end{tikzpicture}
}
\end{center}

We now give references for the positions above.  The position of the class of II$_1$ factors elementary equivalent to $\R$ follows from Corollary \ref{Rsuperee}, Example \ref{RIAC}, and Fact \ref{IAClocal}.  The position of the class of e.c. embeddable II$_1$ factors follows from Proposition \ref{ecIAC}. As previously mentioned it is not known if all e.c. embeddable are elementary equivalent (to $\R$).  It is known, however, that there exists a non-e.c. embeddable II$_1$ factor that is elementary equivalent to $\R$ (see \cite[Theorem 3.6]{ECfactors}).  The position of the class of strongly McDuff II$_1$ factors follows from Proposition \ref{superstrong} and Proposition \ref{adrian}.
It is unknown if there is a strongly McDuff AC factor. The position of the class of II$_1$ factors of the form $P \otimes R$ with $P$ Gamma, non-McDuff (in particular that it is disjoint from the class of super McDuff factors) follows from Proposition \ref{superstrong}.  It is unknown if there is a II$_1$ factor of the form $P \otimes R$ with $P$ Gamma, non-McDuff that is also AC.

The position of $L(\mb{F}_2)$ is due to \cite{mvn4}.  The position of $L(K)$ is due to \cite{DL}.  In \cite{sakai0} it was shown that $L(\mb{F}_2)\otimes \R$ is not AC. In \cite{sakai0} it was shown that $\ds \bigotimes_\mb{N} L(\mb{F}_2)$ is AC; in \cite{DL} it was shown that $\ds \bigotimes_\mb{N} L(\mb{F}_2)$ is super McDuff; in \cite{ZM} it was shown that $\ds \bigotimes_\mb{N} L(\mb{F}_2)$ is not IAC; \cite{popagap} shows that $\ds \bigotimes_\mb{N} L(\mb{F}_2)$ is not strongly McDuff; and \cite{marr} or Proposition \ref{itfullsuper} shows that $\ds \bigotimes_\mb{N} L(\mb{F}_2)$ is not of the form $P \otimes \R$ with $P$ Gamma, non-McDuff.  The position of $L(\breve{\mb{F}_2})$ is due to \cite{ZM}, and it is unknown if $L(\breve{\mb{F}_2})$ is elementarily equivalent to $\R$ or e.c. embeddable.  The position of $L((\mathbb F_2\times \mathbb Z)^{\breve{}})$ is due to \cite{ZM}, and it is unknown if $L((\mathbb F_2\times \mathbb Z)^{\breve{}})$ is e.c. embeddable or of the form $P \otimes \R$ with $P$ Gamma, non-McDuff.  In \cite{sakai0} it was shown that $\ds \bigotimes_\mb{N} L(K)$ is AC;  in \cite{ZM} is was shown that $\ds \bigotimes_\mb{N} L(K)$ is not IAC; in \cite{DL} is was shown that $\ds \bigotimes_\mb{N} L(K)$ is not super McDuff; and it is unknown if $\ds \bigotimes_\mb{N} L(K)$ is of the form $P \otimes \R$ with $P$ Gamma, non-McDuff. 


We close this section with some open questions related to the above diagram.

\begin{question}

\ 
\begin{enumerate}
    
	\item Is there a II$_1$ factor that is AC and of the form $P \otimes \R$ with $P$ Gamma, non-McDuff? In particular, can an infinite tensor product of II$_1$ factors be of the form $P \otimes \R$ with $P$ Gamma, non-McDuff?
	
	\item Does there exist an infinite tensor product of non-Gamma II$_1$ factors that is also IAC?
	
\end{enumerate}
\end{question}


\section{Transferring the property of having factorial commutant}\label{Tran}

\subsection{The transfer theorem}

In order to state the main theorem of this subsection, we need a few model-theoretic definitions.

\begin{defn}
A formula $\varphi(x)$ is called \textbf{existential} (resp. \textbf{$\exists_2$}) if it is of the form $\inf_x \psi(x)$ (resp. $\inf_x\sup_y \psi(x,y)$) where $\psi$ is quantifier-free.
\end{defn}

\begin{defn}
Suppose that $i:N\hookrightarrow M$ is an embedding.  We say that $i$ is \textbf{existential} if, for any existential formula $\varphi(x)$ and any $a\in N$, $\varphi(a)^M=\varphi(a)^N$.
\end{defn}

We will often abuse terminology and refer to a formula as existential (resp. $\exists_2$) if it is equivalent in any tracial von Neumann algebra to an existential (resp. $\exists_2$) formula.

\begin{remark}

\

\begin{enumerate}
    \item $N$ is e.c. (embeddable) if and only if any embedding $i:N\hookrightarrow M$ (with $M$ embeddable) is existential.
    \item The embedding $i:N\hookrightarrow M$ is existential if and only if, whenever $\varphi$ is a nonnegative existential formula and $a\in N$ is such that $\varphi(i(a))^M=0$, then $\varphi(a)^N=0$.
\end{enumerate}
\end{remark}

\begin{defn}
Suppose that $i:N\hookrightarrow M$ is an embedding.  We say that $i$ is \textbf{downward $\exists_2$} if, for any nonnegative $\exists_2$ formula $\varphi(x)$ and any $a\in N$, if $\varphi(a)^M=0$, then $\varphi(a)^N=0$.
\end{defn}

A careful examination of the proof of Fact \ref{braddisaac}(1) establishes the following theorem.  For the reader's convenience, we reproduce the proof here.

\begin{thm}\label{transfer}
Suppose that $i:N\hookrightarrow M$ and $j:M\hookrightarrow P$ are embeddings.
\begin{enumerate}
    \item Suppose that $j$ is existential and $M$ is separably saturated (e.g. when $M$ is an ultraproduct).  If $(j\circ i)(N)'\cap P$ is a factor, then so is $i(N)'\cap M$.
    \item Suppose that $j$ is downwards $\exists_2$ and that $M$ and $P$ are separably saturated.  Then if $i(N)'\cap M$ is a factor, then so is $(j\circ i)(N)'\cap P$.
\end{enumerate}
\end{thm}

\begin{proof}
For (1), fix a countable sequence $(a_k)$ from the unit ball of $N$ that generates $N$.  Suppose that $(j\circ i)(N)'\cap P$ is a factor.  Fix $b\in Z(i(N)'\cap M)$; we wish to show that $b\in \mathbb C$.  To do this, it suffices to show that $j(b)\in Z((j\circ i)(N)'\cap P)$.  It is clear that $j(b)\in (j\circ i)(N)'\cap P$.  Suppose, towards a contradiction, that there is $c\in (j\circ i)(N)'\cap P$ such that $\|[j(b),c]\|_2=\epsilon>0$.  Then for any $m$, we have that
$$\left(\inf_x\max\left(\max_{1\leq k\leq m}\|[x,(j\circ i)(a_k)]\|_2,\epsilon\dotminus \|[j(b),x]\|_2\right)\right)^{P}=0.$$ Since $j$ is existential and $M$ is separably saturated, there is $d\in M$ such that $[d,i(a_k)]=0$ for $k=1,\ldots,m$ and such that $\|[b,d]\|_2\geq \epsilon$.  By separable saturation again, there is $d\in M$ such that $[d,i(a_k)]=0$ for all $k\in \mathbb N$ and such that $\|[b,d]\|_2\geq \epsilon$, contradicting the fact that $b\in Z(i(N)'\cap M)$. 

We now prove (2).  We argue by contrapositive.  Suppose $a\in Z((j\circ i)(N)'\cap P)\setminus \mathbb C$.  Set $\epsilon:=d(b,\tr(b)\cdot 1)>0$.  For the sake of simplicity, we suppose that $N$ is singly generated, say by $w\in N$.\footnote{This is merely a matter of convenience.  In the general case, we would have to explain the legitimacy of treating a countable weighted sum of formuale as a formula in its own right.}  Since $P$ is separably saturated, by \cite[Proposition 7.14]{mtfms}, there is a continuous, nondecreasing function $\alpha:\mathbb R\to \mathbb R$ with $\alpha(0)=0$ such that
$$\left(\sup_y \left(\|[y,a]\|_2\dotminus \alpha(\|[(j\circ i)(w),y]\|_2\right)\right)^P=0,$$
whence $$\left(\inf_x\max(\|[x,(j\circ i)(w)]\|_2,\sup_y \left(\|[y,x]\|_2\dotminus \alpha(\|[(j\circ i)(w),y]\|_2\right),\epsilon\dotminus d(x,\tr(x)\cdot 1)\right)^P=0.$$ Since the displayed formula is $\exists_2$, we have
$$\left(\inf_x\max(\|[x,i(w)]\|_2,\sup_y (\|[y,x]\|_2\dotminus \alpha(\|[i(w),y]\|_2),\epsilon\dotminus d(x,\tr(x)\cdot 1)\right)^M=0.$$ Since $M$ is separably saturated, there is $c\in M$ such that $$\left(\max(\|[c,i(w)]\|_2,\sup_y (\|[y,c]\|_2\dotminus \alpha(\|[i(w),y]\|_2),\epsilon\dotminus d(c,\tr(c)\cdot 1)\right)^M=0.$$  We claim that $c\in Z(i(N)'\cap M)$.  Indeed, since $\|[c,i(w)]\|_2=0$, we have that $c\in i(N)'\cap M$.  Moreover, if $d\in i(N)'\cap M$, then $\alpha(\|[i(w),d]\|_2)=0$, whence $\|[d,c]\|_2=0$.  Since $\epsilon\dotminus d(c,\tr(c)\cdot 1)=0$, we have that $i(N)'\cap M$ is not a factor.
\end{proof}



\subsection{Applications of the transfer theorem}

Before proceeding with applications of the transfer theorem, we first need to observe:

\begin{lem}\label{ultrapres}
Suppose that $i_k:N_k\hookrightarrow M_k$ is a sequence of maps and let $i:\prod_\u N_k\hookrightarrow \prod_\u M_k$ be the induced map, that is, $i((a_k)_\u)=(i_k(a_k))_\u$ for all $(a_k)_\u\in \prod_\u N_k$.  If each $i_k$ is existential (respectively downwards $\exists_2$), then so is $i$.
\end{lem}

\begin{proof}
Set $N:=\prod_\u N_k$ and $M:=\prod_\u M_k$.  Let $\varphi(x)$ be a nonnegative existential (resp. $\exists_2$) formula and $a\in N$ be such that $\varphi(a)^M=0$.  Fix $\epsilon>0$.  Take $I\in \u$ such that $\varphi(i_k(a_k))^{M_k}\leq \epsilon$ for $k\in \u$.  Set $\psi=\varphi\dotminus \epsilon$, which is also existential (resp. $\exists_2$).  For $k\in I$, since $\psi(i_k(a_k))^{M_k}=0$, it follows that $\psi(a_k)^{N_k}=0$, whence $\psi(a)^N=0$, that is, $\varphi(a)^N\leq \epsilon$.  Since $\epsilon>0$ was arbitrary, we have that $\varphi(a)^N=0$ as desired.
\end{proof}



\begin{cor}\label{factorthru}
Suppose that $N$ and $M_k$ are embeddable II$_1$ factors and there are embeddings $i:N\hookrightarrow \R^\u$ and $j:\R^\u\to \prod_\u M_k$ such that $j$ is induced by a sequence of embeddings $j_k:\R\hookrightarrow M_k$.  Further suppose that $(j\circ i)(N)'\cap \prod_\u M_k$ is a factor.  Then $i(N)'\cap \R^\u$ is a factor.  
\end{cor}

\begin{proof}
Since each $M_k$ is embeddable, we have that each $j_k$ is existential.  By Lemma \ref{ultrapres}, $j$ is existential.  The result now follows from Theorem \ref{transfer}(1). 
\end{proof}

\begin{cor}
Under the same assumptions as in the previous corollary, if $(N,\R)$ is also a generalized Jung pair, then $N\cong \R$.
\end{cor}

\begin{cor}
Suppose that $N$ and $M_k$ are embeddable II$_1$ factors and every embedding $N\hookrightarrow \prod_\u M_k$ has factorial relative commutant.  Then $N\cong \R$.
\end{cor}

\begin{proof}
The assumptions and Corollary \ref{factorthru} imply that $(N,\R)$ is a strong factorial commutant pair.  Consequently, $N\cong \R$ by Theorem \ref{Relcom}.
\end{proof}



\begin{cor}
Suppose that $M$ is embeddable.  Then $(N,M)$ is a strong factorial commutant pair if and only if $N\cong \R$.
\end{cor}

\begin{cor}\label{McDufforembeddable}
Suppose that $N$ is embeddable, $M$ is either McDuff or embeddable, and $(N,M)$ is both a factorial commutant pair and a generalized Jung pair.  Then $N\cong \R$.
\end{cor}

\begin{cor}
Suppose that $N$ is an embeddable factor that does not have property Gamma.  Suppose also that $M$ is a II$_1$ factor that is either McDuff or embeddable.  Then $(N,N\otimes M)$ is not a generalized Jung pair.
\end{cor}

\begin{proof}
By \cite[Appendix Corollary]{cartan}, $N$ has w-spectral gap in $N\otimes M$.  In particular, we have
$$N'\cap (N\otimes M)^\u=(N'\cap (N\otimes M))^\u=M^\u,$$ which is a factor.  Consequently, $(N,N\otimes M)$ is a factorial commutant pair.  Thus, by Corollary \ref{McDufforembeddable}, $(N,N\otimes M)$ is not a generalized Jung pair. 
\end{proof}

The previous results were about applying the ``downwards'' part of the transfer theorem.  Our next series of applications involve ``upwards'' transfer.  First, we need the following:

\begin{lem}\label{ecE2}
Suppose that $N\subseteq M$ and both $N$ and $M$ are e.c. (embeddable).  Then the inclusion map is downwards $\exists_2$.
\end{lem}

\begin{proof}
Suppose $(\inf_x\sup_y\varphi(a,x,y))^M=0$, where $\varphi$ is quantifier-free and $a\in N$.  Fix $\epsilon>0$ and let $b\in M$ be such that $(\sup_y\varphi(a,b,y))^M<\epsilon$.  Since $N$ is e.c. (embeddable), there is a map $j:M\hookrightarrow N^\u$ that whose restriction to $N$ is the diagonal embedding $N\hookrightarrow N^\u$.  Since $M$ is e.c. (embeddable), we have $(\sup_y \varphi(a,j(b),y))^{N^\u}\leq \epsilon$, whence $(\inf_x\sup_y\varphi(a,x,y))^{N^\u}\leq \epsilon$.  It follows that $(\inf_x\sup_y\varphi(a,x,y))^N\leq \epsilon$.  Since $\epsilon$ was arbitrary, the result follows. 
\end{proof}

\begin{cor}
Suppose that $(N,\R)$ is a factorial commutant pair.  Then for every e.c. embeddable factor $M$, we have that $(N,M)$ is also a factorial commutant pair.
\end{cor}

\begin{proof}
Let $i:N\hookrightarrow \R^\u$ be such that $i(N)'\cap \R^\u$ is a factor and let $j:\R^\u\hookrightarrow M^\u$ be induced by a sequence of embeddings $\R\hookrightarrow M$.  By Lemmas \ref{ultrapres} and \ref{ecE2}, $j$ is downwards $\exists_2$.  By Theorem \ref{transfer}(2), we have that $i^*:=j\circ i$ is such that $i^*(N)'\cap M^\u$ is a factor, whence $(N,M)$ is a factorial commutant pair.
\end{proof}

\begin{cor}
Suppose that $(N,\R)$ is a factorial commutant pair.  If there is an e.c. embeddable factor $M$ such that $(N,M)$ is a generalized Jung pair, then $N\cong \R$.
\end{cor}



The transfer theorem allows us to say something about Question \ref{fingensuper}:

\begin{thm}
Suppose that $N$ is a finitely generic factor.  Then the following are equivalent:
\begin{enumerate}
    \item $N$ is super McDuff.
    \item For any e.c. factor $M$, $(N,M)$ is a strong factorial commutant pair.
    \item There is an e.c. factor $M$ such that $(N,M)$ is a factorial commutant pair.
\end{enumerate}
\end{thm}

\begin{proof}
We first prove (1) implies (2).  Towards that end, fix an embedding $i:N\hookrightarrow M^\u$.  Since $N$ is locally universal, there is an embedding $j_0:M\hookrightarrow N^\u$.  By Downwards L\"oweneim-Skolem, there is a separable elementary substructure $P$ of $N^\u$ such that $j_0(M)\subseteq P$.  Let $j:M^\u\hookrightarrow P^\u$ be the map induced by $j_0$.   Since $M$ is e.c., $j_0$ is existential, whence so is $j$ by Lemma \ref{transfer}.  Since $N$ is generalized Jung, the composition $j\circ i:N\hookrightarrow P^\u$ is elementary.  By Fact \ref{braddisaac}(1) and the fact that $N$ is super McDuff, $(j\circ i)(N)'\cap P^\u$ is a factor.  Since $j$ is existential,  we have that $i(N)'\cap M^\u$ is a factor, as desired.

(2) implies (3) is trivial.  We now prove the implication (3) implies (1).  Suppose that $M$ is an e.c. factor and $i:N\hookrightarrow M^\u$ is such that $i(N)'\cap M^\u$ is a factor.  By Downwards L\"owenhim-Skolem, there is a separable $N_1\equiv N$ such that $M\subseteq N_1$.  Let $j:M^\u\hookrightarrow N_1^\u$ be induced by the inclusion map.  By Lemmas \ref{ultrapres} and \ref{ecE2}, we have that $j$ is downwards $\exists_2$.  By Theorem \ref{transfer}(2), it follows that $(j\circ i)(N)'\cap N_1^\u$ is a factor.  Since $N$ is finitely generic, $j\circ i$ is elementary.  It follows from Fact \ref{braddisaac}(1) that $N$ is super McDuff.
\end{proof}




\bibliographystyle{plain}
\bibliography{jungbib}{}

\end{document}